\documentclass[oneside, a4paper]{amsart}
\usepackage[a4paper,  width=15cm, height=24cm]{geometry}

\usepackage{amsfonts, amsmath, amssymb, amsthm}
\usepackage{mathtools}
\usepackage[latin1]{inputenc}
\usepackage[T1]{fontenc}
\usepackage{lmodern,dsfont}
\usepackage[english]{babel}
\usepackage[stretch=10,shrink=1]{microtype}
\newcommand{\PL}{\mathbb{P}(\mathcal{P})}
\newcommand{\LL}{\mathbb{P}(\mathcal{L})}
\newcommand{\ZZ}{\mathcal{Z}}

\newcommand{\lspan}[1]{\langle{#1}\rangle}
\newcommand{\lspann}[1]{\text{span}\{ {#1}\} }
\newtheoremstyle{dotless}{6pt}{18pt}{}{}{\bfseries}{.}{\newline}{}
\theoremstyle{dotless}
\newtheorem{thm}{Theorem}
\newtheorem{defi}[thm]{Definition}
\newtheorem{lem}[thm]{Lemma}

\newtheorem{prop}[thm]{Proposition}
\newtheorem{cor}[thm]{Corollary}
\newtheorem{fact}[thm]{Fact}
\newcommand{\changefont}[3]{\fontfamily{#1} \fontseries{#2} \fontshape{#3} \selectfont}
\changefont{ptm}{m}{n}
\title{Dupin cyclidic systems geometrically revisited}
\author{Gudrun Szewieczek}
\begin{document}

\maketitle
\noindent
\bibliographystyle{plain} 
\begin{center}
\begin{minipage}{13cm}\small
\textbf{Abstract.} The induced metrics of Dupin cyclidic systems, that is, orthogonal coordinate systems with Dupin cyclides and spheres as coordinate surfaces, were provided by Darboux. Here we take a more geometric point of view and discuss how Dupin cyclides and Lam\'e families of Dupin cyclidic systems can be obtained by suitably evolving an initial circle or a Dupin cyclide, respectively.
This approach reveals that those Lam\'e families are given by parallel surfaces in various space forms.  
\end{minipage}
\vspace*{0.5cm}\\\begin{minipage}{13cm}\small
\end{minipage}
\vspace*{0.5cm}\\\begin{minipage}{13cm}\small
\end{minipage}
\end{center}

\section{Introduction}
\noindent Dupin cyclides were introduced in~\cite{dupin1822} as surfaces where all curvature lines are circular. Following~\cite{MR1246529} an even weaker characterization leads to this prominent surface class: Dupin cyclides are, apart from spheres and planes, the only surfaces that can be foliated by two orthogonal families of circular arcs. A vast number of works on Dupin cyclides illustrates the great interest in this surface class and offers a variety of different approaches to these surfaces and their applications (see for example \cite{dupin_cyclides, blending_dupin, blending_dupin_2, PRATT1990221}). 

Hence, it is natural to ask whether there are triply orthogonal systems with Dupin cyclides as coordinate surfaces. Classical geometers found first explicit examples \cite{salkowski} and Darboux computed the induced metrics for these kind of systems \cite{darboux_ortho}. However, a complete geometric understanding for this class of systems is still missing. Since in orthogonal coordinate systems any two coordinate surfaces from different families intersect along curvature lines, all orthogonal trajectories of these systems are circular. Thus, Dupin cyclidic systems (DC-systems) are cyclic systems with respect to all three coordinate directions \cite{rib_cyclic}.
\\\\The main aim of this work is to gain a geometric overview on DC-systems and provide an explicit geometric construction by suitably evolving an initial Dupin cyclide. On the way we also discuss how Dupin cyclides can be generated by a similar approach and show that those arise from the initial data of a curvature circle and a 1-parameter family of M\"obius transformations.  
\\\\Simple but frequently used examples of DC-systems include spherical, bipolar cylindrical and toroidal coordinates (cf.\,\cite{MR947546}). However, the class of DC-systems is considerably richer, which is also reflected in the wide range of applications in various different fields. One of the crucial aspects for the use of DC-systems is that their flat induced metrics are R-separable in the Laplace equation (see~\cite{darboux_ortho} or \cite{sym_sep} for simplified proofs and generalizations to higher dimensions). 

A special subclass of DC-systems, namely focal conic domains (FCDs), plays an important role in liquid cristal theory. It was observed experimentally~\cite{friedel}, as well as via a variational approach~\cite{Schief2005OnAN}, that the layer-like structure of smectic phases is locally provided by confocal Dupin cyclides. Global results in this direction are obtained by considering packings with FCDs, as for example in~\cite{Honglawan34, Lavrentovich5}.

Further, cubes of DC-systems are closely related to discrete triply orthogonal systems: those provide elementary hexahedrons of 3D cyclidic nets (cf.\,\cite{dcyclidic}). In this sense parts of DC-systems can be used to ``smoothen''  discrete triply orthogonal systems.
\\\\\textbf{Preliminary works.} Since Dupin cyclides are isothermic surfaces, the induced flat metrics of DC-systems arise in Darboux's classification of isothermic coordinate systems~\cite{darboux_ortho, sym_sep}. Moreover, in~\cite[Chapter~I\!I\!I]{darboux_ortho} Darboux described the following geometric construction for Lam\'e families of DC-systems, that is, a 1-parameter family of surfaces that is part of a triply orthogonal system: take all circles that intersect a Dupin cyclide and a sphere orthogonally. This circle congruence is then cyclic and the 1-parameter family of surfaces orthogonal to the congruence yields a Lam\'e family of a DC-system. He left it as an open question whether this approach yields all Lam\'e families contained in DC-systems. 
%
%
\\\\\textbf{Methodology.} As discussed in~\cite{blaschke, book_cecil} Dupin cyclides admit an elegant description in Lie sphere geometry, the geometry of oriented spheres introduced by Lie in~\cite{L1872}. In particular, the class of Dupin cyclides is invariant under Lie sphere transformations, those transformations that preserve oriented contact between spheres. However, since in general Lie sphere transformations do not preserve angles, when discussing DC-systems, we have to switch to a M\"obius geometric setup. Joyfully, by fixing a point sphere complex, we can recover a M\"obius geometry as subgeometry of Lie sphere geometry and can still use some of the powerful and elegant tools of Lie sphere geometry.   

The underlying idea of this work is to construct Dupin cyclides as well as Lam\'e families of DC-systems by evolving an initial circular curvature line or a Dupin cyclide, respectively. It turns out that in both cases the evolution map is provided by a 1-parameter family of suitable Lie inversions, the basic transformations in Lie sphere geometry. Hence, this approach gives an efficient and explicit way to construct these objects and gain further insights into the geometry of DC-systems.
\\\\\textbf{Structure.} Section~\ref{sect_prelim} starts with a brief introduction to Lie sphere geometry and some basic results in this realm. In Section~\ref{sect_dupin_cyclides} we follow the approach taken in~\cite{spherical} and demonstrate how a Dupin cyclide is obtained by evolving an initial circular curvature line. This construction enables us to directly construct 2-ortho Dupin cyclides~\cite{ortho_circles} and DC-systems with one family of totally umbilic coordinate surfaces (see Subsections~\ref{subsect_2_ortho} and~\ref{subsect_tos_2_ortho}).

Section~\ref{sect_dcsystems} is then devoted to Dupin cyclidic systems: since  any two Dupin cyclides in a Lam\'e family of a DC-system form a Ribaucour pair, the study of Ribaucour related Dupin cyclides provides the basis for our considerations on DC-systems (cf.\,Cor~\ref{cor_r_cyclides_dupin} and Prop~\ref{prop_dupin_rib}). As one of our main results (Thm~\ref{thm_dc_via_evolution}) we prove that evolving a Dupin cyclide via Lie inversions determined by linear sphere complexes that are represented by elements in a 2-dimensional subspace of $\mathbb{R}^{4,2}$ leads to the sought-after Lam\'e families. As a consequence those are given by parallel surfaces in various space forms (see Prop~\ref{prop_parallel_spaceform}). In particular, for the case of hyperbolic space forms the boundary, considered as totally umbilic surface, also yields a coordinate surface in the Lam\'e family; thus, these are exactly the families that are obtained by Darboux's construction. 

In the last Section~\ref{sect_applications} we briefly sketch how the developed evolution approach could be used for several scopes as visualization of Dupin cyclides and blending problems and point out some relations to  discrete differential geometry.  
\\\\\textbf{Acknowledgements.} The author would like to thank Udo Hertrich-Jeromin, Mar\'ia Lara Mir\'o and Martin Peternell for fruitful and pleasurable discussions around the subject. 
%
%
\section{Preliminaries}\label{sect_prelim}
\noindent In this section we present the basics of Lie sphere geometry and their M\"obius geometric subgeometries. A more extensive introduction to this topic can be found for example in \cite{blaschke, book_cecil, coolidge}.
\\\\Throughout this paper we shall use the hexaspherical coordinate model of Lie sphere geometry as introduced by Lie [27] and consider the 6-dimensional vector space $\mathbb{R}^{4,2}$ equipped with a metric $\lspan{.,.}$ of signature $(4,2)$. 

Any element in the \textit{projective light cone (Lie quadric)}
\begin{equation*}
\LL = \{ \lspann{\mathfrak{s}} \ | \ \lspan{\mathfrak{s},\mathfrak{s}}=0, \ \mathfrak{s} \neq 0 \} \subset \mathbb{P}(\mathbb{R}^{4,2})
\end{equation*}
will be identified with an oriented 2-sphere. Note that also oriented planes and points (spheres with radius zero) are considered as oriented 2-spheres. Homogeneous coordinates of elements in $\mathbb{P}(\mathbb{R}^{4,2})$ will be denoted by corresponding black letters; if statements hold for arbitrary homogeneous coordinates we will use this convention without explicitly mentioning it.
\\\\Two spheres $s_1, s_2 \in \LL$ are in oriented contact if and only if $\lspan{\mathfrak{s}_1, \mathfrak{s}_2}= 0$. If two spheres $s_1$ and $s_2$ are in oriented contact, then the elements in $\lspann{s_1, s_2}$ represent a \textit{contact element}, that is, a pencil of spheres in oriented contact. The set of all contact elements, hence all lines in $\LL$, will be denoted by $\mathcal{Z}$.

Surfaces in this model are described by \textit{Legendre maps}, that is, by a 2-parameter family of contact elements, $f:M^2 \rightarrow \mathcal{Z}$, satisfying the contact condition $\lspan{\mathfrak{s_1}, d\mathfrak{s_2}}=0$ for any $s_1, s_2 \in f$.
\\\\The transformation group in this framework is comprised by \textit{Lie sphere transformations}, those are transformations that map spheres to spheres such that oriented contact between spheres is preserved. It is important to note that in general Lie sphere transformations do not preserve the angle between two spheres and can interchange spheres, points and planes.
\\\\\textbf{M\"obius subgeometries.}A M\"obius geometry appears as subgeometry of Lie sphere geometry by choosing a point sphere complex~$\mathfrak{p} \in \mathbb{R}^{4,2}$, $\lspan{\mathfrak{p}, \mathfrak{p}}=-1$. The group of \textit{M\"obius transformations} is then provided by the Lie sphere transformations that fix this point sphere complex~$\mathfrak{p}$ and therefore preserve the set of \textit{point spheres} (spheres with radius zero)
\begin{equation*}
\PL:=\{ v \in \LL \ | \ \lspan{\mathfrak{v}, \mathfrak{p}}=0 \}.
\end{equation*}
Moreover, these transformations preserve the (unoriented) angle $\varphi$ between two spheres $s_1, s_2 \in \LL$ that is given by
\begin{equation*}
\cos \varphi = 1 - \frac{\lspan{\mathfrak{s}_1,\mathfrak{s}_2} \lspan{\mathfrak{p},\mathfrak{p}}}{\lspan{\mathfrak{s}_1,\mathfrak{p}}\lspan{\mathfrak{s}_2,\mathfrak{p}}}.
\end{equation*}
In particular, two spheres $s_1 \in \LL$ and $s_2 \in \LL$ intersect \emph{orthogonally} if and only if 
\begin{equation*}
\lspan{\mathfrak{s}_1, \mathfrak{s}_2+\lspan{\mathfrak{s}_2, \mathfrak{p}}\mathfrak{p}}=0.
\end{equation*}
Note that as a special case we say that a point sphere intersects a sphere orthogonally if and only if the point sphere lies on the sphere. 
%
%
\\\\\textbf{Linear sphere complexes and Lie inversions.} In Lie sphere geometry,
any element $a \in \mathbb{P}(\mathbb{R}^{4,2})$ defines
a \emph{linear sphere complex}
$\mathbb{P}(\mathcal{L}\cap\{a\}^\perp)$,
that is, a 3-dimensional family of 2-spheres.
We distinguish three types of linear sphere complexes:
\begin{itemize}
\item if $\lspan{ \mathfrak{a}, \mathfrak{a}} = 0$, the complex
is called \emph{parabolic} and consists of all spheres
that are in oriented contact with the sphere represented
by~$a$;
\item if $\lspan{ \mathfrak{a}, \mathfrak{a}} < 0$, we
say that the complex is \emph{hyperbolic}; if such a complex is chosen as a point sphere complex, it contains all point spheres of the distinguished M\"obius subgeometry;
\item if $\lspan{\mathfrak{a}, \mathfrak{a}} > 0$, we obtain an \emph{elliptic}
linear sphere complex. In a M\"obius subgeometry modelled on $\lspann{\mathfrak{p}}^\perp$ we have the following geometric
characterization: the linear sphere complex contains
all spheres that intersect the two spheres (that coincide
up to orientation)
\begin{equation}\label{equ_sphere_compl}
 \mathfrak{s}_a^\pm \in \lspann{\mathfrak{a}, \mathfrak{p}}
\end{equation}
at the constant angle
\begin{equation*}
 \cos^2 \varphi = \frac{K}{K-1}, \ \ \text{where } \ 
 K = \frac{\lspan{\mathfrak{a}, \mathfrak{p}}^2}
  {\lspan{\mathfrak{a}, \mathfrak{a}}
   \lspan{\mathfrak{p}, \mathfrak{p}}}.
\end{equation*}
\end{itemize}
\noindent Any elliptic and hyperbolic linear sphere complex may be
used to define a reflection:
let $a \in \mathbb{P}(\mathbb{R}^{4,2})$, $\lspan{
\mathfrak{a}, \mathfrak{a}} \neq 0$, then the \emph{Lie
inversion with respect to the linear sphere complex
$\mathbb{P}(\mathcal{L}\cap\{a\}^\perp)$} is given by
\begin{equation*}
 \sigma_a:\mathbb{R}^{4,2} \rightarrow \mathbb{R}^{4,2}, \ \ 
 \mathfrak{r} \mapsto \sigma_a(\mathfrak{r}) :
  = \mathfrak{r}-\frac{2\lspan{\mathfrak{r}, \mathfrak{a}}}
   {\lspan{\mathfrak{a}, \mathfrak{a}}}\mathfrak{a}.
\end{equation*}
According to the linear sphere complex, the Lie inversion is said to be \emph{elliptic} or \emph{hyperbolic}.

Any Lie inversion $\sigma_a$ is an involutory Lie sphere transformation that preserves all elements lying in the corresponding linear
sphere complex $\mathbb{P}(\mathcal{L}\cap\{a\}^\perp)$. 
\\\\We recall the following geometric standard construction for Lie inversions:
\begin{fact}\label{fact_fourspheres_lie}
Let $s_1, s_2, s_3$ and $s_4 \in \LL$ be four spheres that are mutually not in oriented contact. Then there exists a unique Lie inversion $\sigma_a$ such that
\begin{equation*}
\sigma_a(s_1)=s_2 \ \ \text{and } \ \sigma_a(s_4)=s_3.
\end{equation*} 
If homogeneous coordinates for the spheres satisfy the linear combination $\mathfrak{s}_1-\mathfrak{s}_2+\mathfrak{s}_3-\mathfrak{s}_4=0$, then the linear sphere complex $a$ is explicitly given by $\mathfrak{a}:= \mathfrak{s}_1-\mathfrak{s}_2= \mathfrak{s}_4-\mathfrak{s}_3$.
\end{fact}
\noindent\textbf{M-Lie inversions.} Once a M\"obius geometry is fixed, there are distinguished Lie inversions that belong to its underlying transformation group: a Lie inversion $\sigma_a$ that preserves the point sphere complex
will be called an \emph{M-Lie inversion} (cf.\,\cite{blaschke, hertrichjeromin2021discrete}).
\\\\Clearly, any M-Lie inversion is a M\"obius
transformation and generalizes the concept of M\"obius
inversions: if $a$ determines an elliptic linear sphere
complex, the M-Lie inversion becomes a proper M\"obius inversion,
that is, it provides a reflection in the spheres $s^\pm_a$
as given in (\ref{equ_sphere_compl}). We remark that in this case $\sigma_a(s_a^\pm)=s_a^\mp$. However, if the
corresponding linear sphere complex is hyperbolic,
it can be thought of as an antipodal map.

Note that the M-Lie inversion $\sigma_\mathfrak{p}$ with
respect to the point sphere complex $\mathfrak{p}$ reverses
the orientation of all spheres.

\begin{fact}\label{fact_minv_fix}
Let $\sigma_a$ be an M-Lie inversion that maps the sphere $s_1$ onto the sphere $s_2$. Then the following holds:
\begin{itemize}
\item $\sigma_a$ fixes all spheres that are in oriented contact with $s_1$ and $s_2$, as well as all spheres that are orthogonal to $s_1$ and $s_2$;
\item the corresponding linear sphere complex is given by 
\begin{equation*}
\mathfrak{a}:=\lspan{\mathfrak{s}_2, \mathfrak{p}} \mathfrak{s}_1 - \lspan{\mathfrak{s}_1, \mathfrak{p}} \mathfrak{s}_2.
\end{equation*}
\end{itemize}
\end{fact}

%
\noindent\textbf{Circles and orthogonal spheres.} While in Lie sphere geometry the notion of a circle does not exist, once a M\"obius subgeometry is fixed, a circle becomes a well-defined notion. Thus, in what follows, we again fix a point sphere complex $\mathfrak{p}$.

We will describe a \emph{circle} $\Gamma$ by an orthogonal splitting of $\mathbb{R}^{4,2}$,
\begin{equation}\label{equ_circle}
\Gamma=(\gamma, \gamma^\perp) \in G_{(2,1)}^\mathcal{P} \times G_{(2,1)},  
\end{equation}
where $G_{(2,1)}^\mathcal{P}$ denotes the set of all $(2,1)$-planes that are orthogonal to the point sphere complex $\mathfrak{p}$; thus, all spheres contained in $\gamma \in G_{(2,1)}^\mathcal{P}$ are point spheres, namely, the points lying on the circle. These point spheres then lie on the spheres in $\gamma^\perp$.
\\\\In what follows, we recall some facts on spheres and circles that intersect orthogonally. Detailed proofs of these statements can be found for example in \cite{coolidge} or \cite{hertrichjeromin2021discrete}.

Firstly, note that a circle $\Gamma=(\gamma, \gamma^\perp)$
intersects a sphere $s \in \LL$ orthogonally if and only if
$s$ intersects all spheres in $\gamma^\perp$ orthogonally. Furthermore:

\begin{fact}\label{fact_circle_two_all}
If a circle intersects two elements of a contact element orthogonally, then it intersects all spheres of the contact element orthogonally. 
\end{fact}
\noindent Given a sphere $s \in \LL$ and two point spheres $m_1, m_2 \perp s$ on the sphere $s$. Then the circle $\Gamma=(\gamma, \gamma^\perp)$ that goes through $m_1$ and $m_2$ and intersects the sphere~$s$ orthogonally is given by the $(2,1)$-plane
\begin{equation}\label{equ_orth_circle}
\gamma:=\lspann{\mathfrak{m}_1, \mathfrak{m}_2, \mathfrak{s}+ \lspan{\mathfrak{s}, \mathfrak{p}}\mathfrak{p}}.
\end{equation}
\begin{figure}
\includegraphics[scale=0.2]{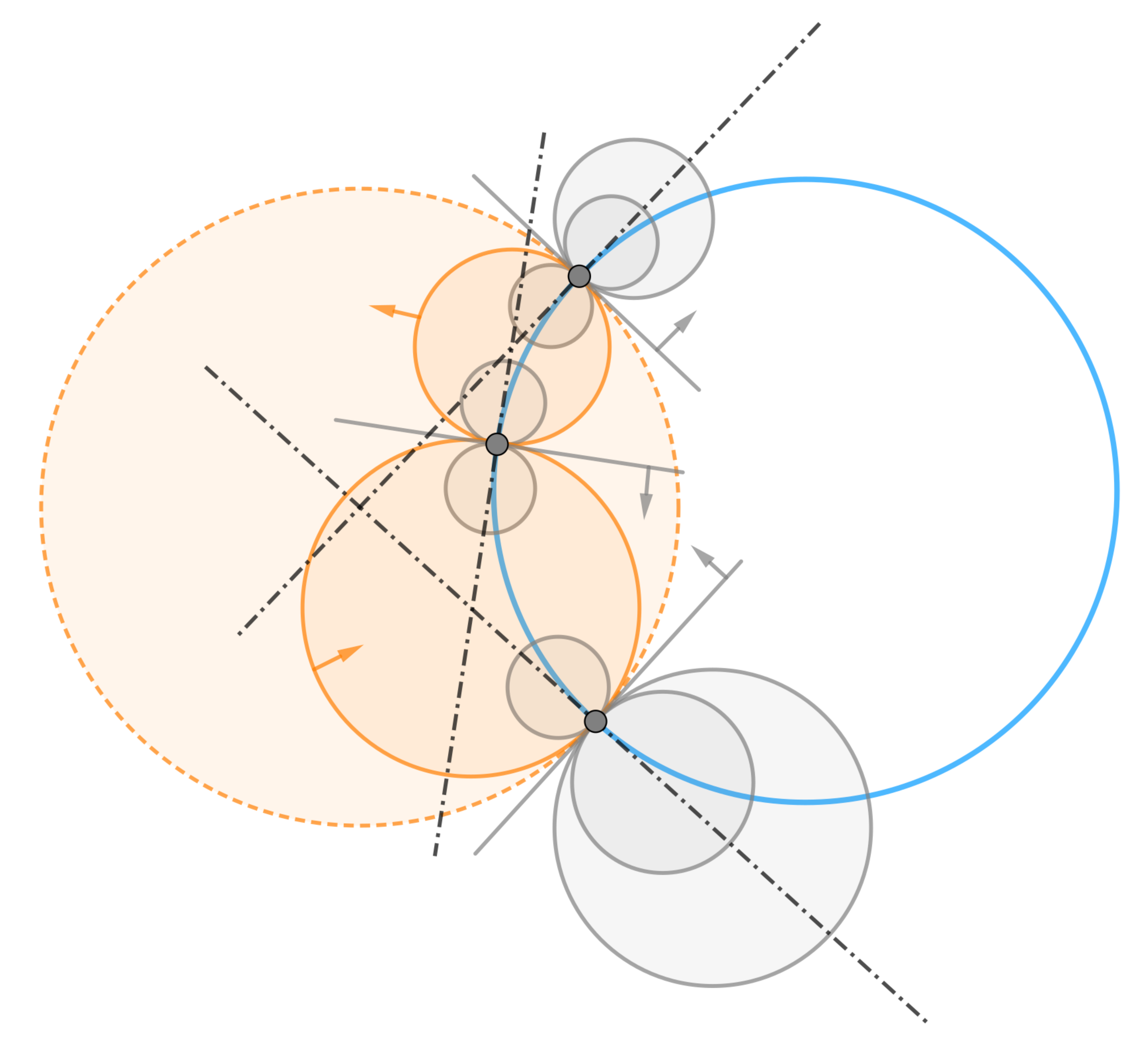}
\caption{Any two contact elements (gray) that intersect a circle (blue) orthogonally share, up to orientation, a common sphere (orange).}\label{fig_orth_cont_el}
\end{figure}
\noindent The spheres that are orthogonal to a fixed circle satisfy the following property (see Figure \ref{fig_orth_cont_el}):
\begin{fact}\label{fact_orth_spheres}
Let $\Gamma$ be a circle, then the spheres that intersect $\Gamma$ in a fixed point $\mathfrak{m} \in \gamma$ of the circle lie in two contact elements $m \in f_m, \tilde{f}_m \in \mathcal{Z}$ that coincide up to orientation. 

Moreover, for any two points $\mathfrak{m}, \mathfrak{n} \in \gamma$ of the circle, two of the orthogonal contact elements share a common sphere, that is,
\begin{equation*}
f_m \cap f_n \neq \emptyset \ \ \ \text{ or } \ \  f_m \cap \sigma_{\mathfrak{p}}(f_n) \neq \emptyset.
\end{equation*}
\end{fact}
\noindent Thus, the spheres orthogonal to a fixed circle lie in a 4-dimensional subspace that can be constructed in the following way:
\begin{lem}\label{lem_orth_contact_el}
Let $f=\lspann{r,s}$ and $\hat{f}=\lspann{r, \hat{s}}$ be two contact elements that share the common sphere~$r$. Then the spheres orthogonal to the circle that intersects the sphere $r$ in the two point spheres of $f$ and $\hat{f}$ orthogonally are the spheres in the subspace $\mathcal{S}:=\lspann{\mathfrak{r}, \mathfrak{s}, \hat{\mathfrak{s}}, \mathfrak{p}} \subset \mathbb{R}^{4,2}$.
\end{lem}
\begin{proof} Let $p_1 \in f \cap \PL$ and $\hat{p}_1 \in \hat{f} \cap \PL$ denote the two point spheres in the contact elements $f=\lspann{r,s}$ and $\hat{f}=\lspann{r, \hat{s}}$. By (\ref{equ_orth_circle}), the circle $\Gamma=(\gamma, \gamma^\perp)$ through $p_1$ and $\hat{p}_1$ that intersects $r = f \cap \hat{f}$ orthogonally is provided by 
\begin{equation*}
\gamma=\lspann{\mathfrak{p}_1, \hat{\mathfrak{p}}_1, \mathfrak{r}+ \lspan{\mathfrak{r}, \mathfrak{p}}\mathfrak{p}}
\end{equation*}
and, thus, the circle points of the indicated circle indeed lie in $\mathcal{S}=\lspann{\mathfrak{r}, \mathfrak{s}, \hat{\mathfrak{s}}, \mathfrak{p}}$.

Further, suppose that $v \in \gamma^\perp \cap \LL$, then $\mathfrak{v}+\lspan{\mathfrak{v},\mathfrak{p}}\mathfrak{p} \in \mathcal{S}^\perp$ and we conclude that any sphere in~$\mathcal{S}$ is orthogonal to the circle~$\Gamma$.

Moreover, since the circle points of~$\Gamma$ and all spheres of the contact element~$f$ lie in $\mathcal{S}$, by Fact~\ref{fact_orth_spheres}, all spheres orthogonal to~$\Gamma$ are contained in this subspace.
\end{proof}
%
%
\noindent\\\textbf{M\"obius geometric sphere pencils.} Let $s_1, s_2 \in \LL$ be two spheres such that $\mathfrak{s}_1 \notin \lspann{\mathfrak{s}_2, \mathfrak{p}}$, then the spheres in the 3-dimensional space
\begin{equation*}
\mathcal{M}:=\lspann{\mathfrak{s}_1, \mathfrak{s}_2, \mathfrak{p}} \subset \mathbb{R}^{4,2}
\end{equation*}
define an \emph{M-sphere pencil}. Depending on the position of the two spheres, we obtain three types:
\begin{itemize}
\item \emph{$0$-pencil}: $s_1$ and $s_2$ intersect in a circle, then all spheres in $\mathcal{M}$ contain the points of this circle of intersection;
\item \emph{$1$-pencil}: $s_1$ and $s_2$ touch, then the spheres in $\mathcal{M}$ yield two contact elements that coincide up to orientation
\item \emph{$2$-pencil}: $s_1$ and $s_2$ do not intersect in a circle or a point, then none of the spheres in the pencil $\mathcal{M}$ do.
\end{itemize}
Observe that if a sphere is orthogonal to the two spheres $s_1$ and $s_2$, then it is also orthogonal to $\sigma_p(s_1)$ and, hence, orthogonal to all spheres in the M-sphere pencil.
%
\section{Dupin cyclides}\label{sect_dupin_cyclides}
\noindent Dupin cyclides were introduced by Dupin in his dissertation in 1803 as surfaces with two families of circular curvature lines~\cite{dupin1822}. Equivalently, Dupin cyclides are characterized by the existence of two curvature sphere congruences that degenerate to two 1-parameter families of spheres. Hence, a Dupin cyclide can be understood as a channel surface with respect to two coordinate directions.
The latter characterization allows an elegant description of Dupin cyclides in the framework of  Lie sphere geometry \cite{blaschke, book_cecil, Pinkall1985}: the vectors representing the spheres of each curvature sphere congruence lie in the intersection of a $(2,1)$-plane with the Lie quadric. Thus, any Dupin cyclide is given by an orthogonal $(2,1)$-splitting $D_1 \oplus_\perp D_2$ of $\mathbb{R}^{4,2}$.

This description is invariant under Lie sphere transformations and, moreover, reveals that any two Dupin cyclides are Lie equivalent.
\\\\Any projection of a Dupin cyclide (represented as a Legendre map) to a M\"obius geometry modelled on $\lspann{\mathfrak{p}}^\perp$ will provide a point sphere map with two families of circular curvature lines. Since in this work we are mainly interested in orthogonal systems, a notion that belongs to M\"obius geometry, we are mainly interested in those projections of Dupin cyclides. Thus, in what follows we will fix a point sphere complex~$\mathfrak{p}$ to distinguish a M\"obius subgeometry. 
\subsection{Parametrizations of Dupin cyclides via M-Lie inversions} \label{subset_param}
In \cite{spherical} Legendre maps with a family of spherical curvature lines were constructed with the help of a Lie sphere geometric evolution. In particular, it was discussed that any such surface is constituted by a prescribed initial spherical curve and a certain 1-parameter family of Lie sphere transformations. 

In this subsection we show that for Dupin cyclides the evolution map simplifies: any Dupin cyclide can be constructed by evolving an initial circle by a suitable 1-parameter family of M-Lie inversions.
\\\\Suppose that 
\begin{equation*}
\Delta= D_1 \oplus_\perp D_2 \in G_{(2,1)} \times G_{(2,1)}
\end{equation*}
represents a Dupin cyclide with the two 1-parameter families of curvature spheres 
\begin{equation*}
I^{1} \ni u \mapsto s^1(u) \in D_1 \cap \LL \ \ \text{and } \  I^{2} \ni  v \mapsto s^2(v) \in D_2 \cap \LL.
\end{equation*}
A curvature sphere $s$ of $\Delta$ is said to be \emph{regular} (with respect to the point sphere complex $\mathfrak{p}$), if it is not a point sphere: $\lspan{\mathfrak{s},\mathfrak{p}} \neq 0$; otherwise it is called \emph{singular}. 

This description of Dupin cyclides also includes circles, namely if the point sphere complex  $\mathfrak{p}$ is contained in one of the $(2,1)$-planes $D_1$ or $D_2$. Then one of the curvature sphere congruences consists of point spheres, the points of the circle (cf.\,(\ref{equ_circle})).

Apart from this special case, the number of singular curvature spheres of a Dupin cyclide is limited: there are at most two point spheres in one of the two curvature sphere congruences. In our studies we sometimes have to exclude these singularities and therefore fix the following notation 
\begin{equation*}
I^{i\star}:=\{t \in I^i \ | \ \lspan{\mathfrak{s}^i(t), \mathfrak{p}} \neq 0\}.
\end{equation*} 
\ \\A parametrization of the Legendre map of $\Delta$ that respects the curvature leaves is provided by
\begin{equation*}
I^1 \times I^2  \ni (u,v) \mapsto \lspann{\mathfrak{s}^1(u),\mathfrak{s}^2(v)} \in \mathcal{Z}.
\end{equation*}
Away from singularities, a curvature line parametrization of the point sphere map of $\Delta$ can be obtained by the following Lie geometric evolution: let $s^i_0:=s^i(t_0) \in D_i$ be a regular curvature sphere of $\Delta$ and consider the 1-parameter family of M-Lie inversions $I^{i\star} \ni t \mapsto \sigma^i_{t}$ with respect to the linear sphere complexes defined by
\begin{equation}\label{equ_evolution_dupin}
  I^{i\star} \ni  t \mapsto \mathfrak{a}^i_t:=\lspan{\mathfrak{s}^i(t),\mathfrak{p}}\mathfrak{s}^i_0 - \lspan{\mathfrak{s}^i_0,\mathfrak{p}}\mathfrak{s}^i(t).
\end{equation} 
This 1-parameter family of M-Lie inversions will be called an  \textit{evolution map of $\Delta$ based at $t_0$}. 

Note that $\mathfrak{a}_t^i \in D_i \cap \lspann{\mathfrak{p}}^\perp$, hence the vectors $\mathfrak{a}_t^i$ that induce the evolution map lie in a 2-dimensional subspace of $\mathbb{R}^{4,2}$. Moreover, this evolution yields a parametrization of the related curvature sphere congruence via $t \mapsto \sigma^i_t(s^i_0)$.
\\\\Further, suppose that the map $I \ni r \mapsto c_0(r) \in G_{(2,1)}^\mathcal{P}$ parametrizes the points of the curvature circle that lies on the curvature sphere $s^i_0 \in D_i$. Then, reflection of this circle~$c_0$ via the M-Lie inversions $t \mapsto \sigma^i_{t}$ provides a 1-parameter family of parametrized circles via
\begin{equation}
I \times I^{i\star} \ni (r,t) \mapsto f(r,t):=\sigma^i_{t}(c_0 (r)) \in \PL.
\end{equation}
Since the M-Lie inversions $\sigma^i_{t}$ interchange the curvature spheres in $D_i$ and fix the curvature spheres in $D_j$, $i \neq j$, we conclude that $f$ indeed provides a curvature line parametrization of $\Delta$.
\\\\The evolution map also transports further geometric data related to the Dupin cyclide along the surface: for example, consider the quer-spheres of a Dupin cyclide.

Suppose that $I^{i\star} \ni t \mapsto \mathfrak{s}^i(t)$, $\lspan{\mathfrak{s}^i,\mathfrak{p}}=1$, parametrizes one of the curvature sphere congruences. For each curvature sphere $s^i(t)$, there exists up to orientation a unique sphere that intersects $s^i(t)$ orthogonally in the corresponding curvature circle. These so-called \textit{quer-spheres}~\cite{blaschke} are provided by
\begin{equation*}
I^{i\star} \ni t \mapsto \mathfrak{q}^{i\pm}(t):= \partial_t\mathfrak{s}^i(t) \pm \sqrt{\lspan{\partial_t\mathfrak{s}^i(t), \partial_t\mathfrak{s}^i(t)}} \mathfrak{p}.
\end{equation*}
For an illustration of the quer-sphere congruences of a Dupin cyclide see the M-sphere pencils in Figure~\ref{fig_dc_evolution_quer}.

Since the M-Lie inversions $t \mapsto \sigma_t^i$ map curvature circles and curvature spheres of $\Delta$ onto each other and preserve orthogonality, we conclude that the evolution map also propagates the quer-spheres $q^{i\pm}$ along the Dupin cyclide
\begin{equation*}
I^{i\star} \ni t \mapsto q^i(t):=\sigma^i_t(q^{i+}(t_0)).
\end{equation*} 
Note that each quer-sphere congruence lies in an M-sphere pencil. The type of sphere pencil depends on the number of singularities in the corresponding curvature sphere family. Detailed configurations of the quer-spheres for the three types of Dupin cyclides $\Delta = D_1 \oplus_\perp D_2$ are summarized in the following table:
\\\begin{center}\begin{tabular}{|p{3cm}|p{9cm}|}
\hline
\textbf{Singularities} & \textbf{quer-sphere congruence $\mathcal{Q}_i$} \\\hline
\end{tabular}
\\[5pt]\begin{tabular}{|p{3cm}|p{9cm}|}
\hline
no sing. in $D_1$& the quer-spheres intersect in a circle (0-pencil)\\\hline
no sing. in $D_2$& the quer-spheres intersect in a circle (0-pencil)\\\hline
\end{tabular}
\\[5pt]\begin{tabular}{|p{3cm}|p{9cm}|}
\hline
$p_1 \in D_1$& $p_1 \in \mathcal{Q}_1$; the quer-spheres lie in  a 1-pencil\\\hline
no sing. in $D_2$& the quer-spheres intersect in a circle (0-pencil)\\\hline
\end{tabular}
\\[5pt]\begin{tabular}{|p{3cm}|p{9cm}|}
\hline
$p_1, p_2 \in D_1$& $p_1, p_2 \in \mathcal{Q}_1$; the quer-spheres lie in a 2-pencil\\\hline
no sing. in $D_2$& the quer-spheres intersect in a circle (0-pencil)\\\hline\end{tabular}
\end{center}
%
\ \\\\Note that, by Fact \ref{fact_minv_fix}, generically the quer-sphere congruences can be equivalently used to describe the evolution maps of~$\Delta$ based at $t_0$: 
\begin{equation*}
I^{i\star} \ni t \mapsto a^i_t=\lspann{\lspan{\mathfrak{q}^i(t),\mathfrak{p}}\mathfrak{q}^i_0 - \lspan{\mathfrak{q}^i_0,\mathfrak{p}}\mathfrak{q}^i(t)}.
\end{equation*}
\noindent Since the spheres of the two quer-sphere congruences mutually intersect at a right angle, this point of view immediately reveals that the M-Lie inversions $\sigma^i_t$ fix all spheres of the quer-sphere congruence $t\mapsto q^{j}(t)$.
%
%
\\\begin{figure}
\begin{minipage}{2cm}
\includegraphics[scale=0.3]{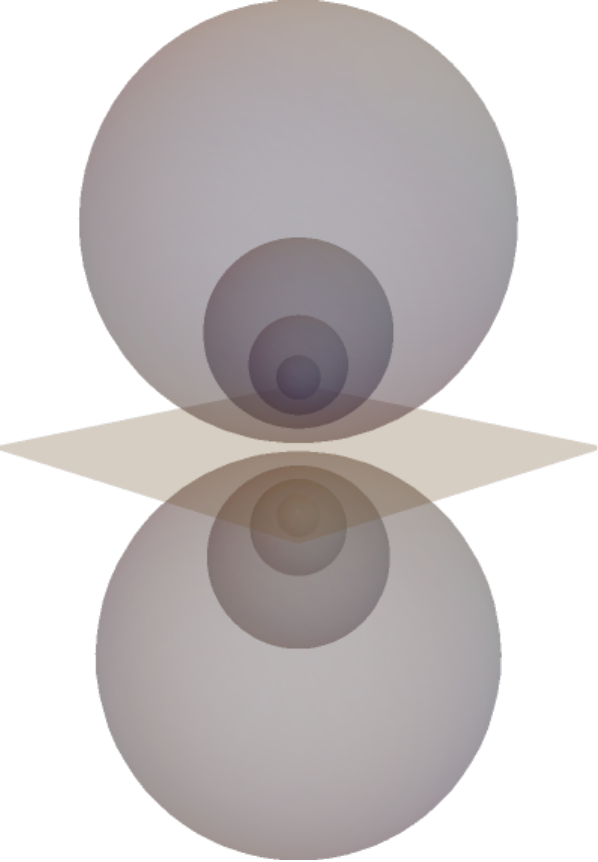}
\end{minipage}
\begin{minipage}{4cm}
\hspace*{0.25cm}\includegraphics[scale=0.4]{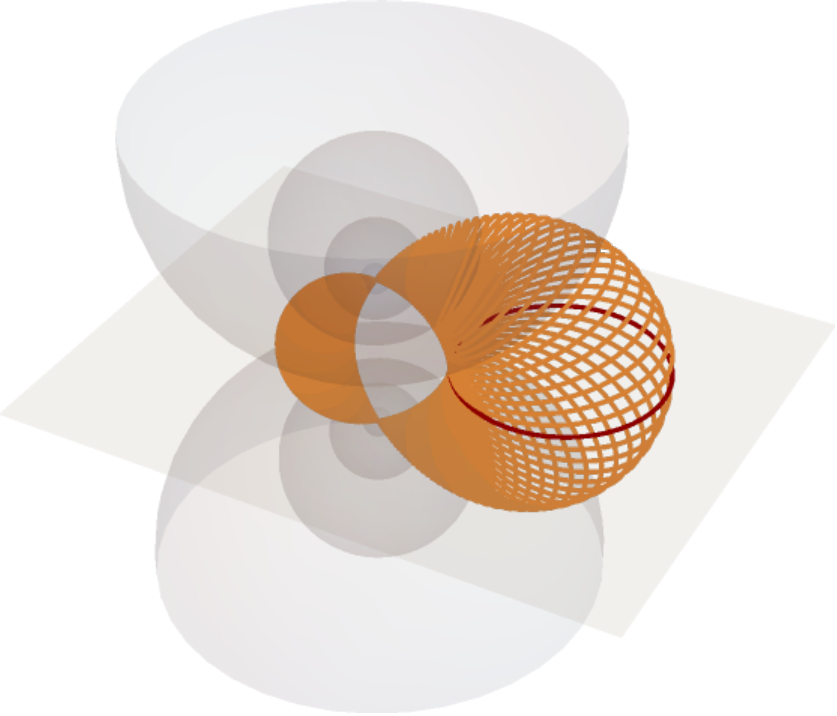}
\end{minipage}
\begin{minipage}{4cm}
\includegraphics[scale=0.4]{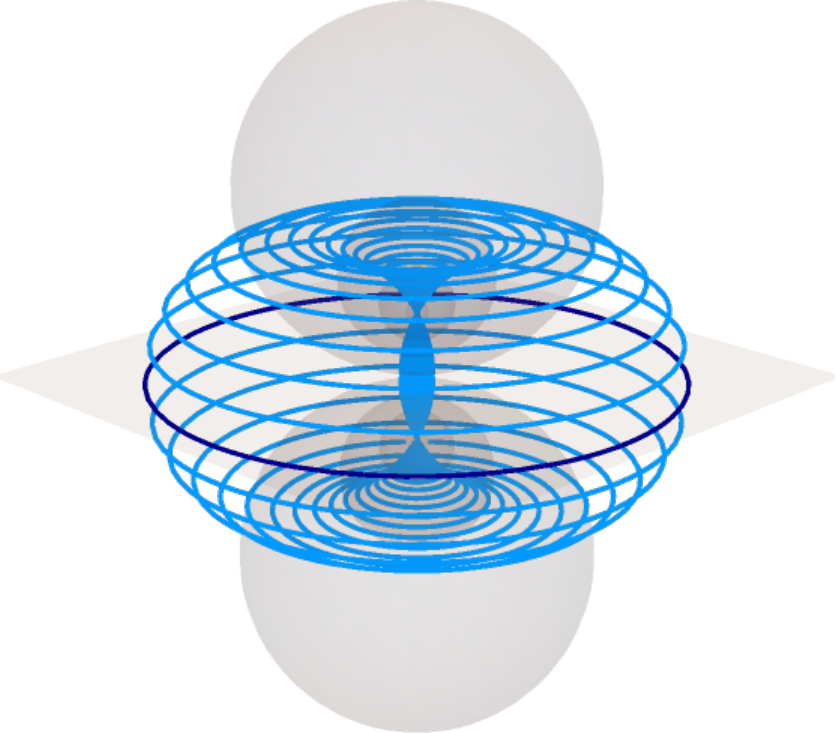}
\end{minipage}
\begin{minipage}{4cm}
\includegraphics[scale=0.35]{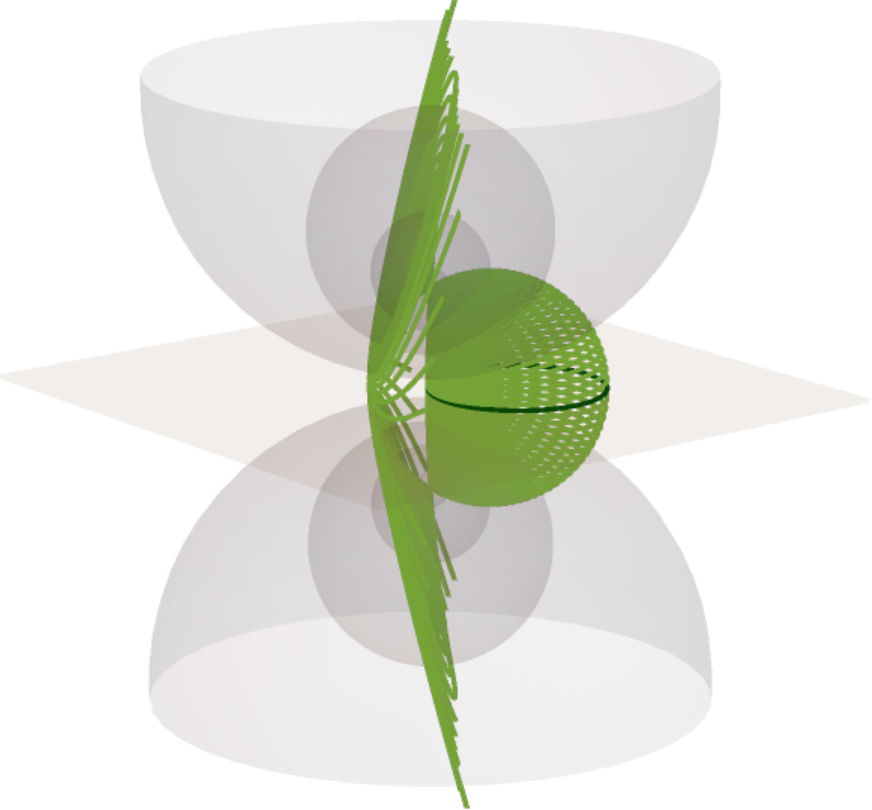}
\end{minipage}
%
\\\vspace*{0.4cm}\begin{minipage}{2cm}
\includegraphics[scale=0.3]{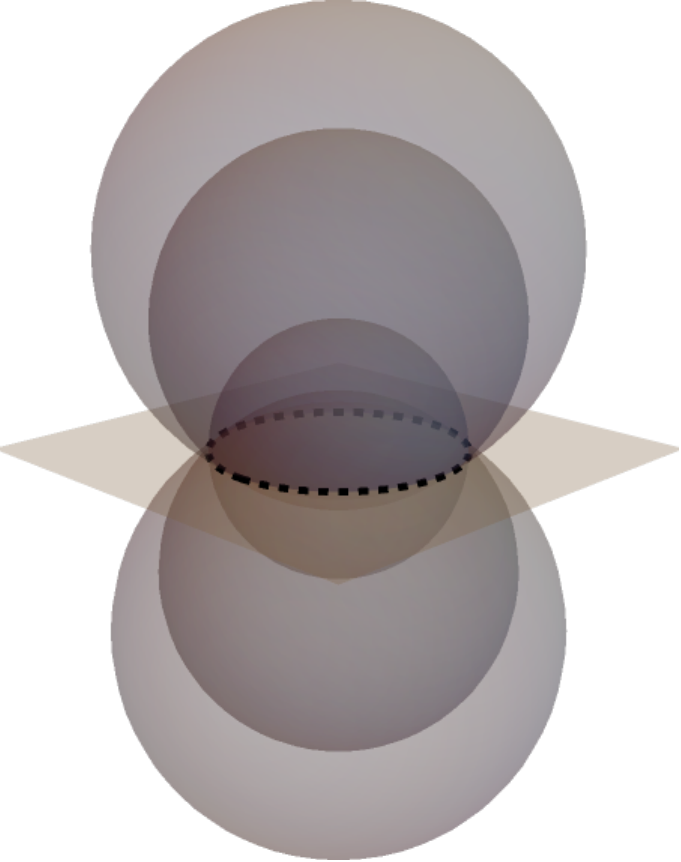}
\end{minipage}
\begin{minipage}{4cm}
\hspace*{0.55cm}\includegraphics[scale=0.4]{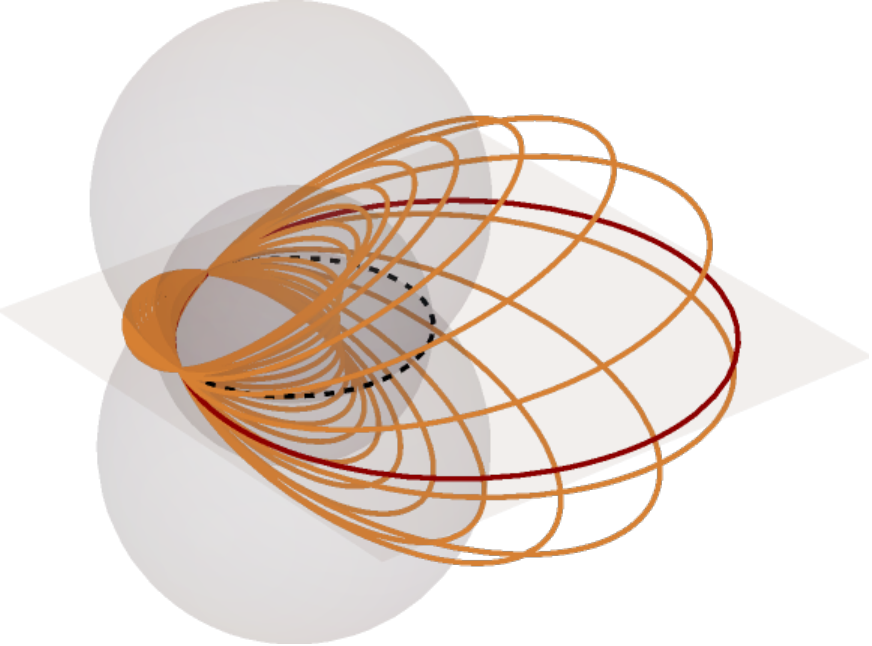}
\end{minipage}
\begin{minipage}{4cm}
\hspace*{0.45cm}\includegraphics[scale=0.45]{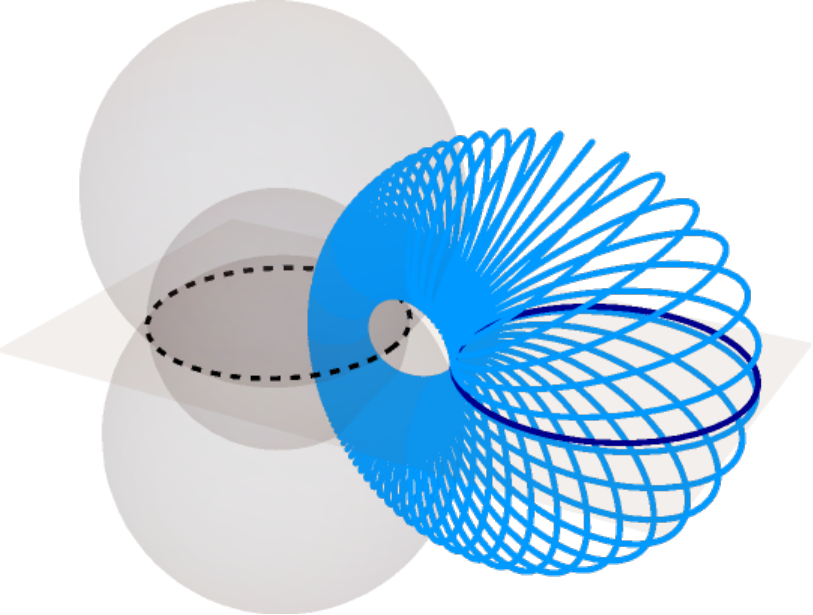}
\end{minipage}
\begin{minipage}{4cm}
\includegraphics[scale=0.45]{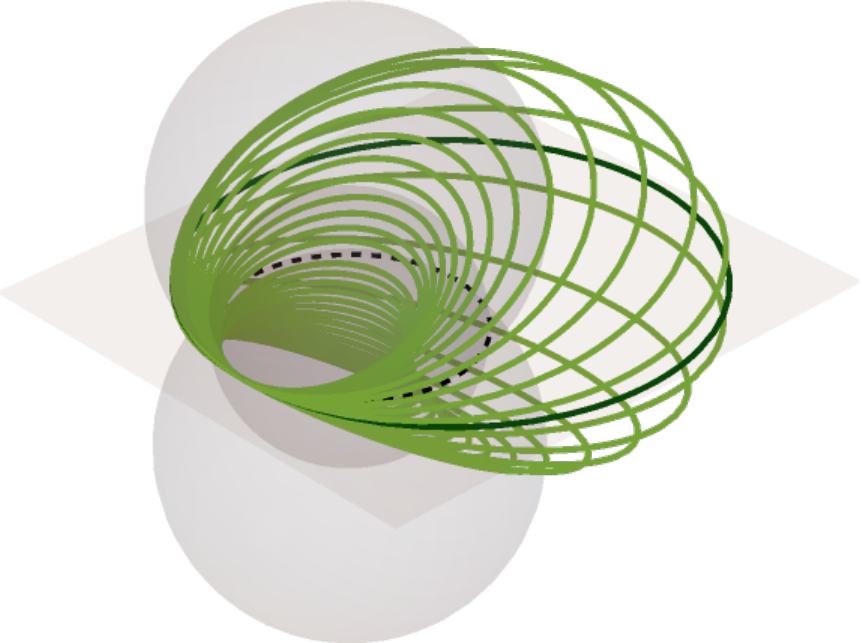}
\end{minipage}
%
\\\vspace*{0.1cm}\begin{minipage}{2cm}
\includegraphics[scale=0.34]{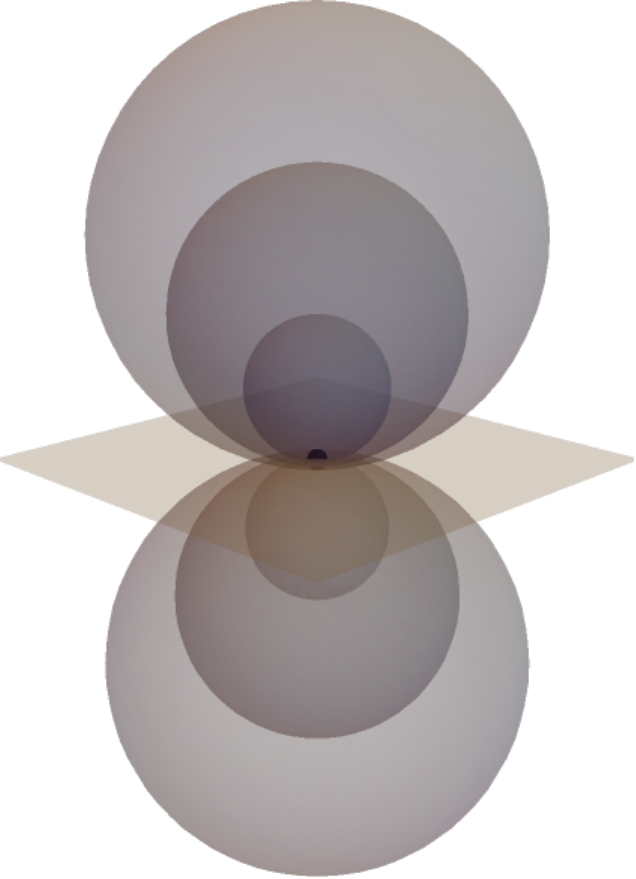}
\end{minipage}
\begin{minipage}{4cm}
\hspace*{0.4cm}\includegraphics[scale=0.45]{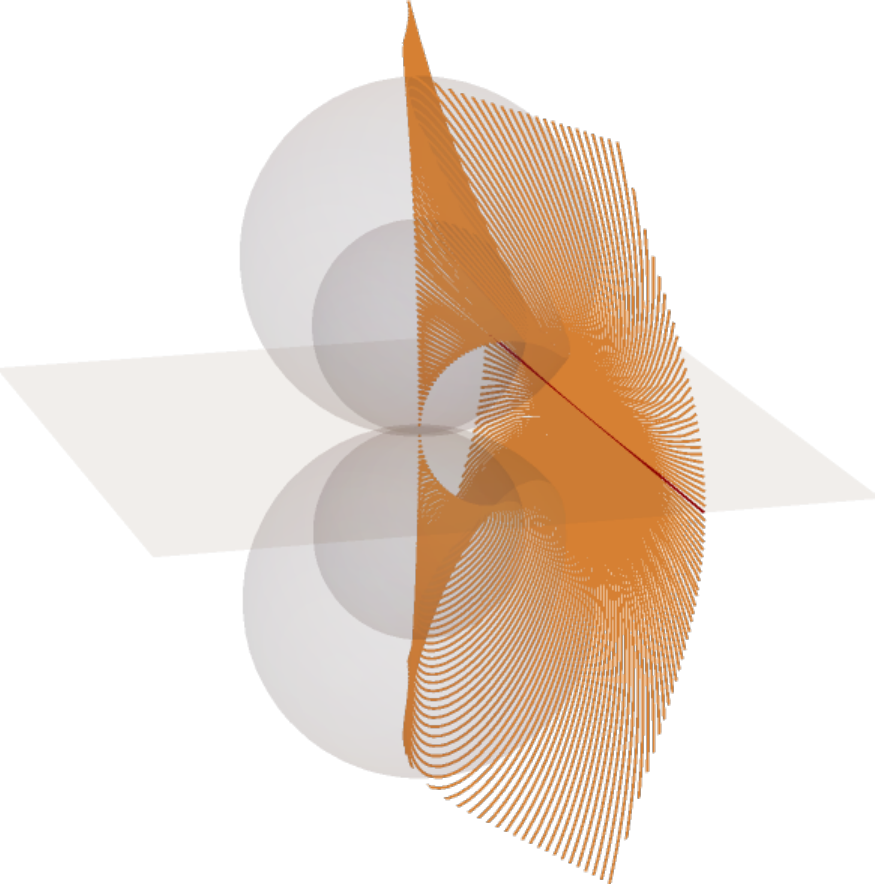}
\end{minipage}
\begin{minipage}{4cm}
\hspace*{0.7cm}\includegraphics[scale=0.45]{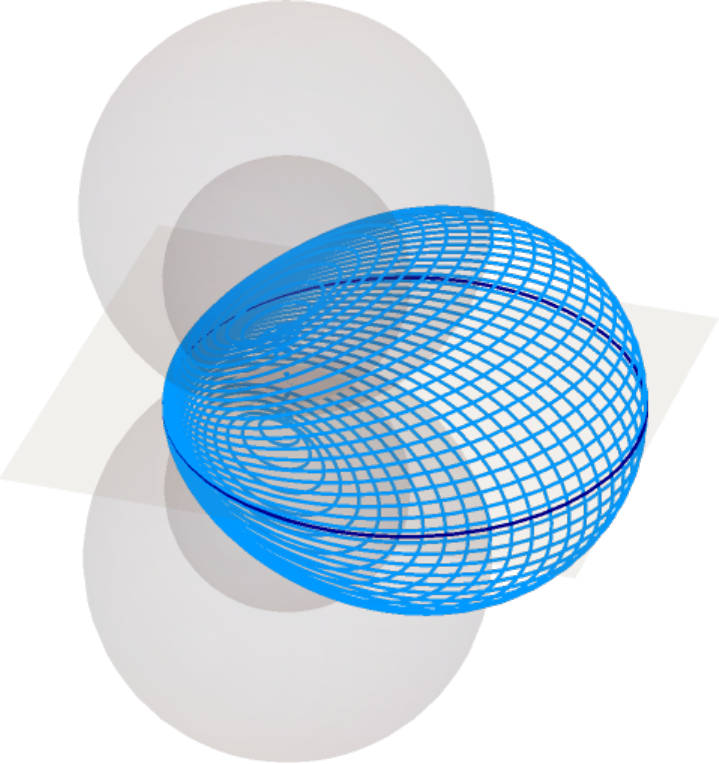}
\end{minipage}
\begin{minipage}{4cm}
\hspace*{0.55cm}\includegraphics[scale=0.4]{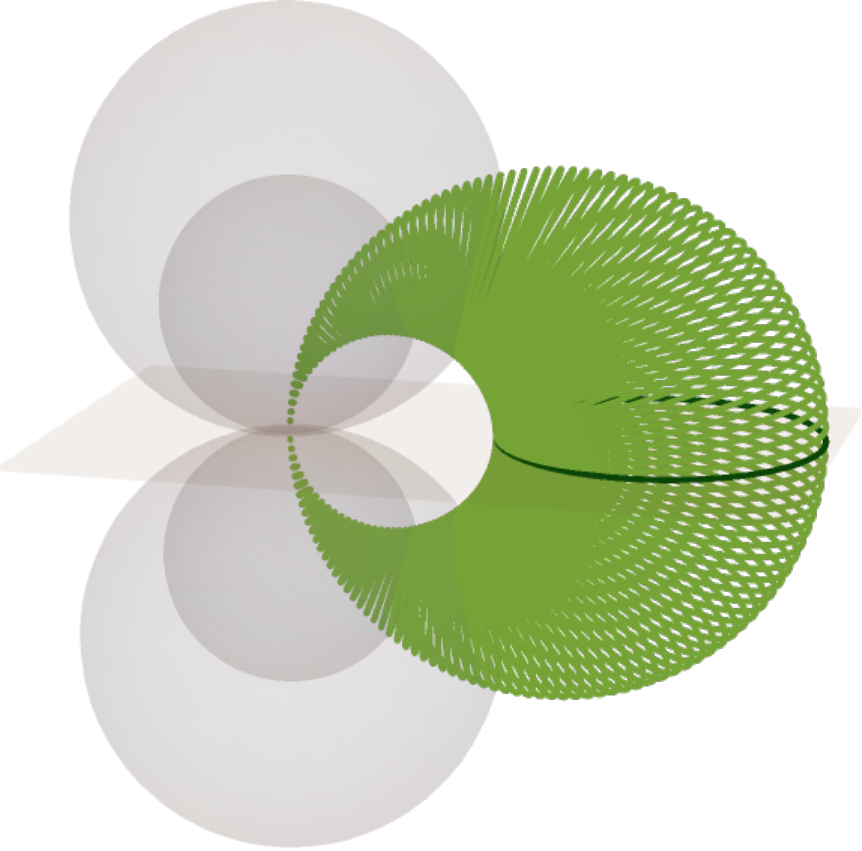}
\end{minipage}
\caption{An M-sphere pencil as prescribed quer-sphere congruence (\emph{left column}) and one initial curvature circle determine (up to orientation) a unique Dupin cyclide. The curvature line parametrization is obtained by evolving the initial circle via the evolution map associated to the M-sphere pencil. The type of quer-sphere pencil and the position of the initial circle determine the number of singularities on the Dupin cyclide.}
\label{fig_dc_evolution_quer}
\end{figure}
\subsection{Dupin cyclides associated to prescribed quer-spheres}\label{subset_quer}
As we have pointed out above, the two quer-sphere congruences of a Dupin cyclide lie in two M-sphere pencils that intersect orthogonally. In this subsection we will discuss that the choice of an M-sphere pencil and one prescribed curvature sphere or curvature circle allows to reverse engineer a Dupin cyclide. 
\\\\Let us start with the simple observation that any M-sphere pencil induces a canonical 1-parameter family of M-Lie inversions: let $J \ni t \mapsto q(t)$ parametrize the oriented spheres in the M-sphere pencil $\mathcal{M}$. For any fixed sphere $q_0:=q(t_0) \in \LL \setminus \PL$ in the M-sphere pencil $\mathcal{M}$ we consider the following linear sphere complexes
\begin{equation*}
J^\star \ni t \mapsto \mathfrak{a}_t := \lspan{\mathfrak{q}(t), \mathfrak{p}} \mathfrak{q}_0 - \lspan{\mathfrak{q}_0, \mathfrak{p}}\mathfrak{q}(t), \ \ \text{where} \ J^\star := \{ t \in J \ |  \ \lspan{\mathfrak{q}_0, \mathfrak{q}(t)} \neq 0 \text{ and } \lspan{\mathfrak{q}(t), \mathfrak{p}} \neq 0 \}.
\end{equation*}
These linear sphere complexes induce a family of M-Lie inversions via $t \mapsto \sigma_{a_t}$; it will be called the \emph{evolution map associated to $\mathcal{M}$ based at $t_0$}.
\begin{prop}\label{prop_evolution_quer1}
Let $t \mapsto q(t)$ parametrize the spheres of an M-sphere pencil $\mathcal{M}$ and denote by $t \mapsto \sigma_t$ the evolution map associated to $\mathcal{M}$ based at $t_0$. For any sphere $s_0 \in \LL$ that is orthogonal to the sphere $q_0:=q(t_0)$ of the M-sphere pencil, the 1-parameter family of spheres 
\begin{equation*}
J^\star \ni t \mapsto s(t):= \sigma_t(s_0)
\end{equation*} 
is either constant or provides a curvature sphere family of a Dupin cyclide.
\end{prop}
\begin{proof}
Let $s_0 \in \LL$ be a sphere orthogonal to $q_{0}$. If $s_0$ is also orthogonal to all other spheres in~$\mathcal{M}$, then due to Fact \ref{fact_minv_fix}, $s_0$ is fixed by all M-Lie inversions in the evolution map and the family $t \mapsto s(t)$ is constant.
\\\\Otherwise, assume that $s_0$ is not orthogonal to all spheres in~$\mathcal{M}$. Since the elements $t \mapsto a_t$ lie in a 2-dimensional subspace of $\mathbb{R}^{4,2}$, we conclude that the spheres $t \mapsto s(t)$ lie in a 3-dimensional subspace of $\mathbb{R}^{4,2}$. Hence, we will show that the spheres~$t \mapsto s(t)$ are in oriented contact with another 1-parameter family of spheres, denoted by $t \mapsto r(t)$, and are therefore indeed curvature spheres of a Dupin cyclide. 

To do so, consider the 1-parameter family $t \mapsto f(t)$ of contact elements that contain the points of the intersection circle of $s_0$ with $q_0$ and that contain the sphere $s_0$. By construction, all spheres in the contact elements are then orthogonal to $q_0$ (cf.\,Fact \ref{fact_circle_two_all}). 

Furthermore, let $q_1$ be another sphere in~$\mathcal{M}$ such that $q_1 \notin \lspann{\mathfrak{q}_0, \mathfrak{p}}$. Then, in each contact element $f(t)$, there exists one sphere $\tilde{r}(t)$ that is also orthogonal to $q_1$: 
\begin{equation*}
\tilde{r}(t):= f(t) \cap \lspann{\mathfrak{q}_1+ \lspan{ \mathfrak{q}_1, \mathfrak{p}}\mathfrak{p}}^\perp.
\end{equation*}
Hence, in particular, the spheres $t \mapsto \tilde{r}(t)$ are orthogonal to all spheres in $\mathcal{M}$. Therefore, due to Fact~\ref{fact_minv_fix}, the spheres $t \mapsto r(t)$ are fixed by all M-Lie inversions of the evolution map associated to~$\mathcal{M}$. Since Lie inversions preserve oriented contact between spheres the claim is proven. 
\end{proof}
\noindent Note that this construction also includes parametrized circles as special Dupin cyclides: if $s_0 \in \PL$ is a point sphere that lies on the initial sphere $q_0$, then $t \mapsto \sigma_t(s_0)$ yields a parametrized circle. 

In particular, if we evolve an initial curvature circle, we obtain curvature line parametrizations of the constructed Dupin cyclides (see Figure~\ref{fig_dc_evolution_quer}):
\begin{cor}\label{cor_evolution_quer2}
Let $u \mapsto \gamma(u) \in \gamma \cap \PL$ be a parametrization of a circle $\Gamma=(\gamma, \gamma^\perp)$ on the initial sphere $q_0$. Then 
\begin{equation*}
(u,t) \mapsto \sigma_t(\gamma(u)) \in \PL
\end{equation*}
provides either a 2-dimensional orthogonal coordinate system of circles on a sphere or a curvature line parametrization of a Dupin cyclide.
\end{cor}
\noindent By construction, the type of the M-sphere pencil determines the number of singular curvature spheres in the obtained Dupin cyclide: if the constructed object is not a constant sphere or a circle, the curvature sphere family obtained by an evolution map associated to an $n$-pencil, $n=\{0,1,2\}$, yields $n$ singular curvature spheres in this family.

Furthermore, it shows that there are three types of Lam\'e families of 2-dimensional orthogonal coordinate systems composed of circles (see also \cite{LEITE20151254}).
%
\subsection{2-ortho-Dupin cyclides} \label{subsect_2_ortho}
In \cite{ortho_circles} circles that intersect a Dupin cyclide in at least two points orthogonally were investigated. It was shown that suitably chosen circles of this type form further Dupin cyclides so that special pairs of orthogonally intersecting Dupin cyclides emerge. The arguments in \cite{ortho_circles} mainly rely on the fact that a Dupin cyclide is M\"obius invariant to a circular cone, a circular cylinder or a torus of revolution, where the situation of orthogonally intersecting circles is rather lucid.

In this subsection we briefly discuss that these pairs of Dupin cyclides that intersect along two common curvature lines orthogonally, can be obtained by a simultaneous evolution. To make this text self-contained, some results of \cite{ortho_circles} will be reproven in the Lie sphere geometric setup. We will then answer some open questions in this realm (cf.\,\cite[\S 7]{ortho_circles}) and discuss in Subsection \ref{subsect_tos_2_ortho} how these pairs of Dupin cyclides arise in special triply orthogonal coordinate systems.

We again fix a point sphere complex $\mathfrak{p}$ and work in the M\"obius subgeometry modelled on $\lspann{\mathfrak{p}}^\perp$.
\begin{defi}
A \emph{2-ortho-circle} of a Dupin cyclide is a circle that intersects the Dupin cyclide in exactly two different points orthogonally. 
\end{defi}
\noindent To avoid case studies, in what follows we only consider 2-ortho-circles that intersect the Dupin cyclide in regular curvature spheres, that is, the points of intersection of the 2-ortho-circle and the Dupin cyclide are not curvature spheres. However, most of the results also hold for these special cases. 

\begin{lem}\label{lem_2-ortho_points}
The two intersection points of a 2-ortho-circle with the Dupin cyclide lie on a common curvature circle. 
\end{lem}

\begin{proof}
Let $\Gamma=(\gamma, \gamma^\perp)$ be a circle that intersects the Dupin cyclide $\Delta=D_1 \oplus_{\perp} D_2$ orthogonally  in the points $m_1$ and $m_2$. We denote the two corresponding contact elements of $\Delta$ through $m_1$ and $m_2$  by $f_{m_i} = \lspann{s_{i1}, s_{i2}}$, $i \in \{ 1, 2\}$ , where $s_{i1} \in D_1$ and $s_{i2} \in D_2$.

There are two possible cases: firstly, suppose that $f_{m_1} \cap f_{m_2} = \emptyset$, then the two further contact elements $f'_{m_1}=\lspann{s_{11}, s_{22}}$ and $f'_{m_2}=\lspann{s_{21}, s_{12}}$ are also part of the Dupin cyclide. Due to Fact \ref{fact_circle_two_all}, the circle $\Gamma$ then  intersects the Dupin cylide in four point spheres orthogonally and is therefore not a 2-ortho circle.

Otherwise, if $f_{m_1} \cap f_{m_2} \neq \emptyset$, the common sphere is a curvature sphere of the Dupin cyclide~$\Delta$ and the two points $m_1$ and $m_2$ therefore lie on a common curvature line, which proves the claim. 
\end{proof}
%
%
\begin{figure}
\includegraphics[scale=0.4]{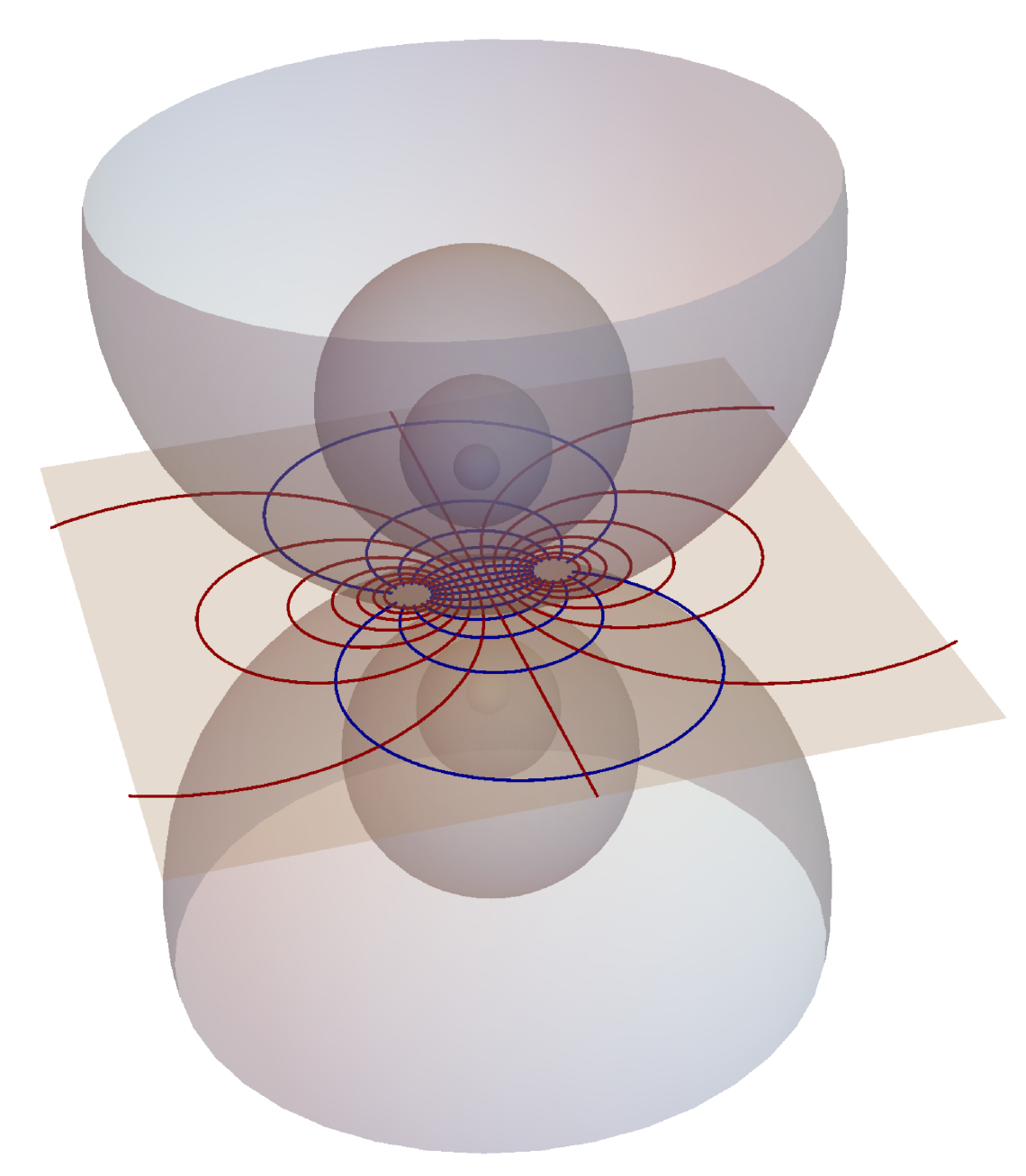}
\includegraphics[scale=0.4]{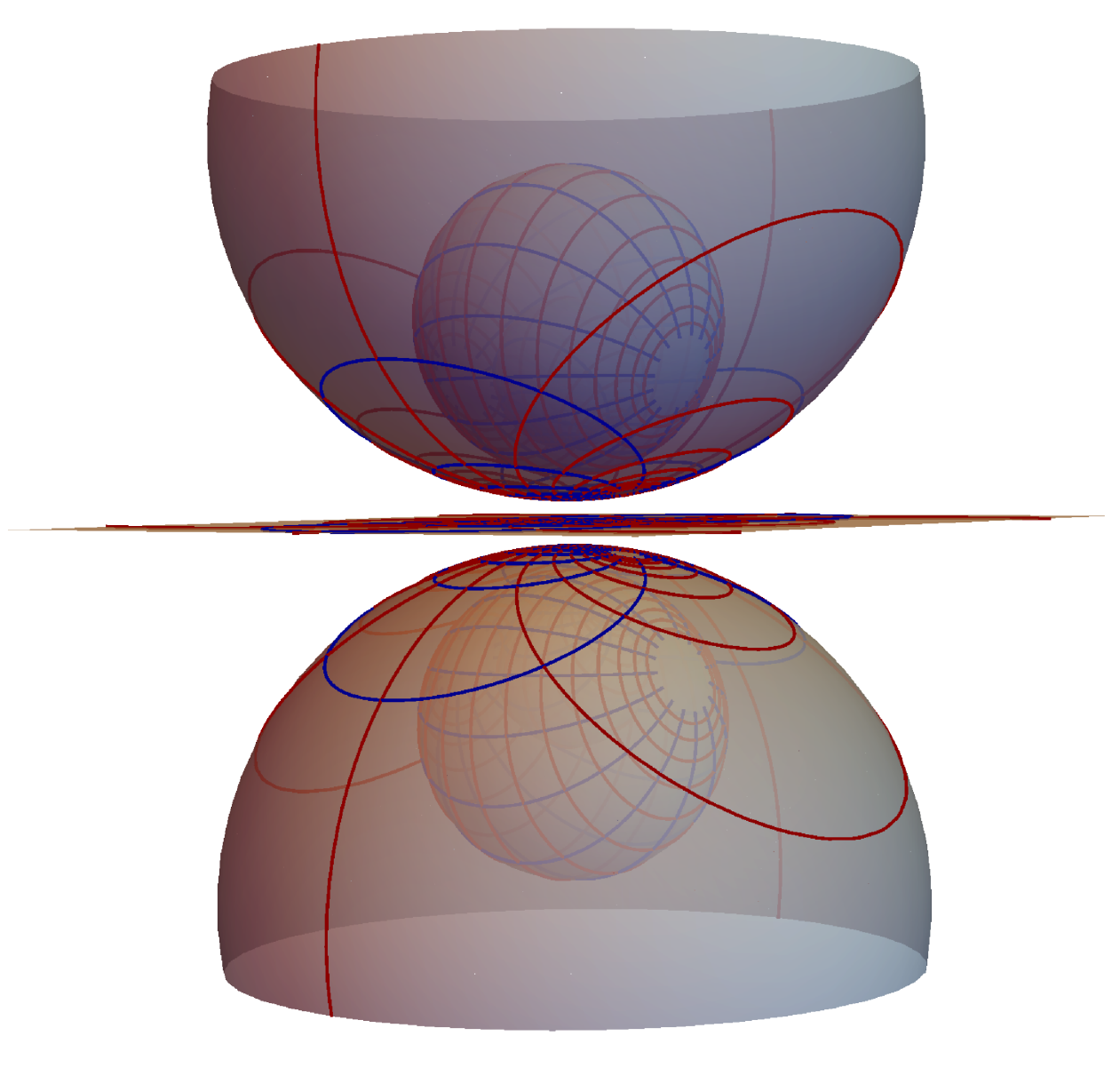}
\caption{Any M-sphere pencil and any 2-dimensional orthogonal coordinate system compound of circles (\emph{left}) on one of the spheres give via associated evolution rise to a triply orthogonal system (\emph{right}). One coordinate surface family consists of the spheres in the M-sphere pencil, the other two generically of Dupin cyclides.}\label{fig_tos}
\end{figure}
%
\noindent Thus, any 2-ortho-circle of a Dupin cyclide can be constructed in the following way:
\begin{itemize}
\item choose two contact elements $f_1$ and $f_2$ of a Dupin cyclide such that its point spheres $m_1$ and $m_2$ lie on a common curvature line, that is, the two contact elements share a common curvature sphere $s:= f_1 \cap f_2$;
\item then, according to (\ref{equ_orth_circle}), the 2-ortho-circle $\Gamma=(\gamma, \gamma^\perp)$ that goes through $m_1$ and $m_2$ is provided by the $(2,1)$-plane
\begin{equation}\label{equ_ortho_circle}
\gamma:=\lspann{\mathfrak{m}_1, \mathfrak{m}_2, \mathfrak{s}+ \lspan{\mathfrak{s}, \mathfrak{p}}\mathfrak{p} } \in G_{(2,1)}^\mathcal{P}.
\end{equation}
\end{itemize}
\ \\This description shows a simple property of 2-ortho-circles that will be crucial in our further studies:
\begin{lem}\label{lem_ortho_on_quer}
Any 2-ortho-circle of a Dupin cyclide lies on a quer-sphere of the Dupin cyclide.
\end{lem}
\begin{proof}
Let $\Delta$ be a Dupin cyclide and consider the 2-ortho circle $\Gamma =(\gamma, \gamma^\perp)$ that intersects $\Delta$ in the two point spheres $m_1$ and $m_2$. From Lemma \ref{lem_2-ortho_points}, we know that $m_1$ and $m_2$ lie on a common curvature circle and therefore on a  curvature sphere; we denote this curvature sphere by $\bar{s}^i$. Hence, by above, the point spheres of this 2-ortho circle lie in the subspace
\begin{equation*}
\gamma=\lspann{\mathfrak{m}_1, \ \mathfrak{m}_2, \ \bar{\mathfrak{s}}^i+ \lspan{\bar{\mathfrak{s}}^i, \mathfrak{p}}\mathfrak{p} } \in G_{(2,1)}^\mathcal{P}.
\end{equation*}

\noindent Moreover, let $\bar{q}^i$ denote a quer-sphere that intersects the curvature sphere $\bar{s}^i$ orthogonally, that is, $\lspan{\bar{\mathfrak{q}}^i, \bar{\mathfrak{s}}^i+\lspan{\bar{\mathfrak{s}}^i, \mathfrak{p}}\mathfrak{p}}=0$.

Since the point spheres $m_1$ and $m_2$ also lie on $\bar{q}^i$, $\lspan{\mathfrak{m}_1, \bar{\mathfrak{q}}^i}=\lspan{\mathfrak{m}_2, \bar{\mathfrak{q}}^i}=0$, we conclude that $\bar{\mathfrak{q}}^i \in \gamma^\perp$ which proves the claim.
\end{proof}
\noindent This immediately shows the following property for all 2-ortho circles that intersect a Dupin cyclide along a fixed curvature circle:
\begin{cor}
All 2-ortho-circles that orthogonally intersect the same curvature circle lie on a common sphere, namely the corresponding quer-sphere. 
\end{cor}
\noindent As already discussed in \cite{ortho_circles}, the 2-ortho circles along two fixed curvature lines of a Dupin cyclide yield another Dupin cyclide. Below we will prove that these two Dupin cyclides can be obtained by using the same evolution map.
\begin{defi}
Let $\Delta$ be a Dupin cyclide, then we call $\Omega$ a \emph{2-ortho Dupin cyclide} of $\Delta$ if it intersects~$\Delta$ orthogonally along two curvature lines. 
\end{defi}
%
\noindent Thus, one of the families of curvature circles of a 2-ortho Dupin cyclide $\Omega$ consists of the 2-ortho circles along two fixed curvature lines of $\Delta$. Hence, by Lemma~\ref{lem_ortho_on_quer}, one of the quer-sphere congruences of the two Dupin cyclides coincides and 2-ortho Dupin cyclides of $\Delta$ are obtained via a simultaneous evolution in the following way:
\begin{itemize}
\item let $\Delta$ be a Dupin cyclide with curvature line parametrization $(u,t) \mapsto \sigma_t(\gamma(u)) \in \PL$ obtained via the evolution map of $\Delta$ based at $t_0$ with initial curvature circle $\Gamma=(\gamma, \gamma^\perp)$ on the initial quer-sphere~$q(t_0)$ (see Corollary~\ref{cor_evolution_quer2}); 
\item the parametrization $u \mapsto \tilde{\gamma}(u) \in \tilde{\gamma} \cap \PL$ of any circle $\tilde{\Gamma}=(\tilde{\gamma}, \tilde{\gamma}^\perp)$ on $q(t_0)$ that intersects the circle $\gamma$ orthogonally gives rise to a 2-ortho Dupin cyclide via
\begin{equation*}
(u,t) \mapsto \sigma_t(\tilde{\gamma}(u)).
\end{equation*}
\end{itemize}
Since the evolution map $t \mapsto \sigma_t$ consists of M-Lie inversions, 2-ortho circles are mapped to 2-ortho circles. Therefore, by Corollary \ref{cor_evolution_quer2}, the proposed construction indeed yields a pair of Dupin cyclides that orthogonally intersect along two common curvature lines.  
%
\begin{cor}\label{cor_2ortho_Dupin}
The 1-parameter family of 2-ortho circles along two fixed curvature lines of the same family of a Dupin cyclide yields curvature circles of a 2-ortho Dupin cyclide.
\end{cor}
\noindent To conclude this subsection, we comment on the existence of a common 2-ortho Dupin cyclide for two prescribed Dupin cyclides:
\begin{cor}
Two Dupin cyclides admit a common 2-ortho Dupin cyclide if and only if they share a common quer-sphere congruence.
\end{cor}
\noindent In particular, if two Dupin cyclides share one 2-ortho Dupin cyclide, then there are infinitely many common 2-ortho Dupin cyclides. Moreover, any two of the involved Dupin cyclides are then related by a Ribaucour transformation (for more details see Subsections \ref{subsect_dc_cyclic} and \ref{subsect_dc_via_evolution}).
%
%
\subsection{Dupin cyclidic systems with a family of totally umbilic coordinate surfaces}\label{subsect_tos_2_ortho} The evolution map of an M-sphere pencil can also be used to construct first examples of Dupin cyclidic systems, that is, triply orthogonal coordinate systems where all orthogonal trajectories are circular.

Suppose that $t \mapsto q(t)$ parametrizes the spheres of the M-sphere pencil $\mathcal{M}$ and denote the associated evolution map of $\mathcal{M}$ based at $t_0$ by $t \mapsto \sigma_t$. Moreover, let $(x,y) \mapsto c(x,y) \in \PL$ be a 2-dimensional orthogonal coordinate system on $q_0:=q(t_0)$ that consists of circular coordinate lines (cf.\,Figure \ref{fig_tos}, left). Then
\begin{equation*}
(x,y,t) \mapsto \sigma_t(c(x,y)) \in \PL
\end{equation*}
yields a triply orthogonal coordinate system with circular coordinate lines (see Figure \ref{fig_tos}, right). The coordinate surface families $x=const$ and $y=const$ are Dupin cyclides or spherical surfaces (cf.\,Corollary \ref{cor_evolution_quer2}). By construction, any two Dupin cyclides of different families form a pair of 2-ortho Dupin cyclides. The third coordinate surface family $t=const$ consists of the spheres of the M-sphere pencil; those are the quer-spheres of the Dupin cyclides in the two families $x=const$ and $y=const$ (for an example see Figure \ref{fig_types_ell}, right).

Well-known examples of this type are (M\"obius transformations of) spherical, bipolar cylindrical and toroidal coordinates \cite{MR947546}.
\section{Dupin cyclidic systems}\label{sect_dcsystems}
\noindent In this section we characterize and geometrically construct triply orthogonal systems where all orthogonal trajectories are circular. Due to Dupin's famous theorem \cite{darboux_ortho, salkowski}, those circles are then curvature lines on the involved coordinate surfaces, which are therefore Dupin cyclides or totally umbilic surfaces \cite{MR1246529}. 

To avoid systems that only consist of totally umbilic coordinate surfaces, as for example the Euclidean coordinates, we make the following definition:
\begin{defi}
A triply orthogonal system where all orthogonal trajectories are circles and at least one Lam\'e family contains Dupin cyclides is called a \emph{Dupin cyclidic system (DC-system)}. 
\end{defi}
\subsection{DC-systems via special cyclic circle congruences}\label{subsect_dc_cyclic}
Classically \cite{darboux_ortho, rib_cyclic, salkowski}, a triply orthogonal system is called \emph{cyclic} if the orthogonal trajectories in one coordinate direction are all circular. This is the case if and only if two coordinate surface families are constituted of channel surfaces, that is, surfaces foliated by one family of circular curvature lines.
 
Recall \cite{rib_cyclic} that any cyclic system may be obtained by the following construction: let $f$ and $\hat{f}$ be two Legendre maps that form a \emph{Ribaucour pair}, that is, the surfaces envelop a common sphere congruence $r$ such that curvature lines correspond. Furthermore, consider the 2-dimensional circle congruence compound of circles that intersect the enveloped sphere congruence orthogonally in the two surface points of contact of $f$ and $\hat{f}$. This circle congruence then admits a 1-parameter family of orthogonal surfaces (which includes also $f$ and $\hat{f}$) that yields a Lam\'e family of a cyclic system. The constructed circles provide the circular orthogonal trajectories of this cyclic system. 

In this way any Ribaucour pair of surfaces gives rise to an associated cyclic system. Clearly, starting with specific Ribaucour pairs of Legendre maps, restricts the geometry of the associated cyclic systems as we will see below for Dupin cyclidic systems.

Since all orthogonal trajectories of a DC-system are circular, those systems are cyclic in all three coordinate directions. In this subsection, we shall discuss how the construction outlined above leads to DC-systems. 
\\\\To start with we recall some facts on Ribaucour pairs of Legendre maps, a well-defined notion in Lie sphere geometry \cite{MR2254053}. Thus, suppose that $f, \hat{f}: M \rightarrow \ZZ$ is a Ribauocur pair simultaneously parametrized along their corresponding curvature leaves such that 
\begin{equation*}
r: M \rightarrow \LL, \ r(u,v):= f(u,v) \cap \hat{f}(u,v)
\end{equation*}
yields the enveloped Ribaucour sphere congruence. As shown in \cite{MR3871108}, the Ribaucour sphere congruence then gives rise to two special congruences of Dupin cyclides, provided by the two families of $(2,1)$-planes $C_i:M \rightarrow G_{(2,1)}$,
\begin{equation*}
C_1=\lspann{\mathfrak{r}, \partial_v\mathfrak{r}, \partial_{vv} \mathfrak{r}} \subset \mathbb{R}^{4,2} \ \ \ \text{and } \ \ C_2=\lspann{\mathfrak{r}, \partial_u\mathfrak{r}, \partial_{uu} \mathfrak{r}} \subset \mathbb{R}^{4,2};
\end{equation*}
those Dupin cyclides are called \emph{Ribaucour cyclides} of the Ribaucour pair. If $s^{i}$ and $\hat{s}^{i}$ denote the curvature sphere congruences of $f$ and $\hat{f}$, then
\begin{equation}\label{equ_R_dupin_perp}
\mathfrak{s}^{1},  \hat{\mathfrak{s}}^{1}, \partial_u \mathfrak{s}^{1}, \partial_u\hat{\mathfrak{s}}^{1} \in C_1^\perp \ \ \text{ and } \ \ \mathfrak{s}^{2},  \hat{\mathfrak{s}}^{2}, \partial_v \mathfrak{s}^{2}, \partial_v\hat{\mathfrak{s}}^{2} \in C_2^\perp. 
\end{equation}
Hence, the Ribaucour cyclides are in first order contact with $f$ and $\hat{f}$.
\\\\We now turn to the Ribaucour pairs that are most relevant for this work, namely those that consist of two Dupin cyclides. Since both curvature sphere congruences of a Legendre map  degenerate to sphere curves $u \mapsto s^{1}(u)$ and $v \mapsto s^{2}(v)$ if and only if the surface is a Dupin cyclide, the situation considerably simplifies (see also \cite[Thm 4.5]{MR3871108}): the Ribaucour cyclide congruences of such a Ribaucour pair also become 1-dimensional: $u \mapsto C_1^\perp(u)$ and $v \mapsto C_2^\perp(v)$.

\begin{cor}\label{cor_r_cyclides_dupin}
A Ribaucour pair consists of two Dupin cyclides if and only if the Ribaucour cyclides are constant along the corresponding curvature directions.
\end{cor}
\noindent This property guarantees that the cyclic systems associated to these types of Ribaucour pairs are DC-systems. We again fix a point sphere complex $\mathfrak{p}$ and work in a M\"obius subgeometry. 

If two corresponding curvature lines of a Ribaucour pair of Dupin cyclides $(\Delta, \hat{\Delta})$ are fixed, the Ribauocur spheres along them are curvature spheres of a fixed Ribaucour cyclide $\Delta_R$. The two fixed curvature lines of $\Delta$ and $\hat{\Delta}$ are then also curvature lines on this Ribaucour cyclide $\Delta_R$. 

Thus, the orthogonal circles associated to the Ribaucour pair $(\Delta, \hat{\Delta})$ along these two fixed curvature lines are 2-ortho circles of the Ribaucour cyclide $\Delta_R$ (see Figure \ref{fig_orth_dupin_system}). By Corollary \ref{cor_2ortho_Dupin}, these 2-ortho circles of $\Delta_R$ then lie on a fixed Dupin cyclide $\bar{\Delta}$. Since the 2-ortho circles belong to the cyclic circle congruence associated to $(\Delta, \hat{\Delta})$, the Dupin cyclide $\bar{\Delta}$ is then a coordinate surface of the associated cyclic system.

Hence, since Dupin cyclides are foliated by circular curvature lines in both directions, all orthogonal trajectories of the cyclic system associated to $(\Delta, \hat{\Delta})$ are circular and this is therefore a DC-system. 

In summary, we have proven:
\begin{thm}\label{thm_dc_via_cyclic}
The cyclic system associated to a Ribaucour pair of Dupin cyclides yields a DC-system. 
\end{thm}

\begin{figure}
\hspace*{-1cm}\includegraphics[scale=0.4]{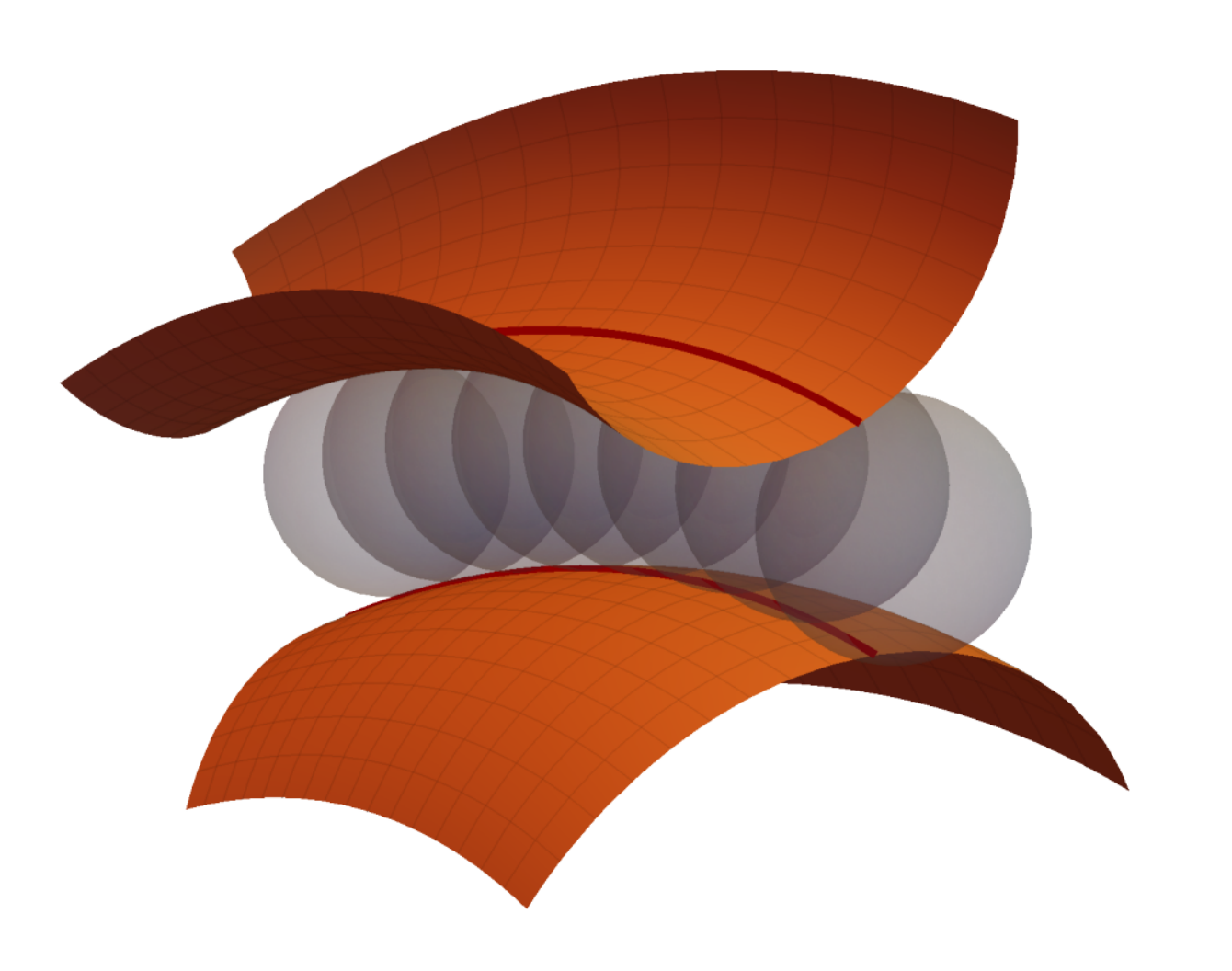}
\hspace*{0.5cm}\includegraphics[scale=0.4]{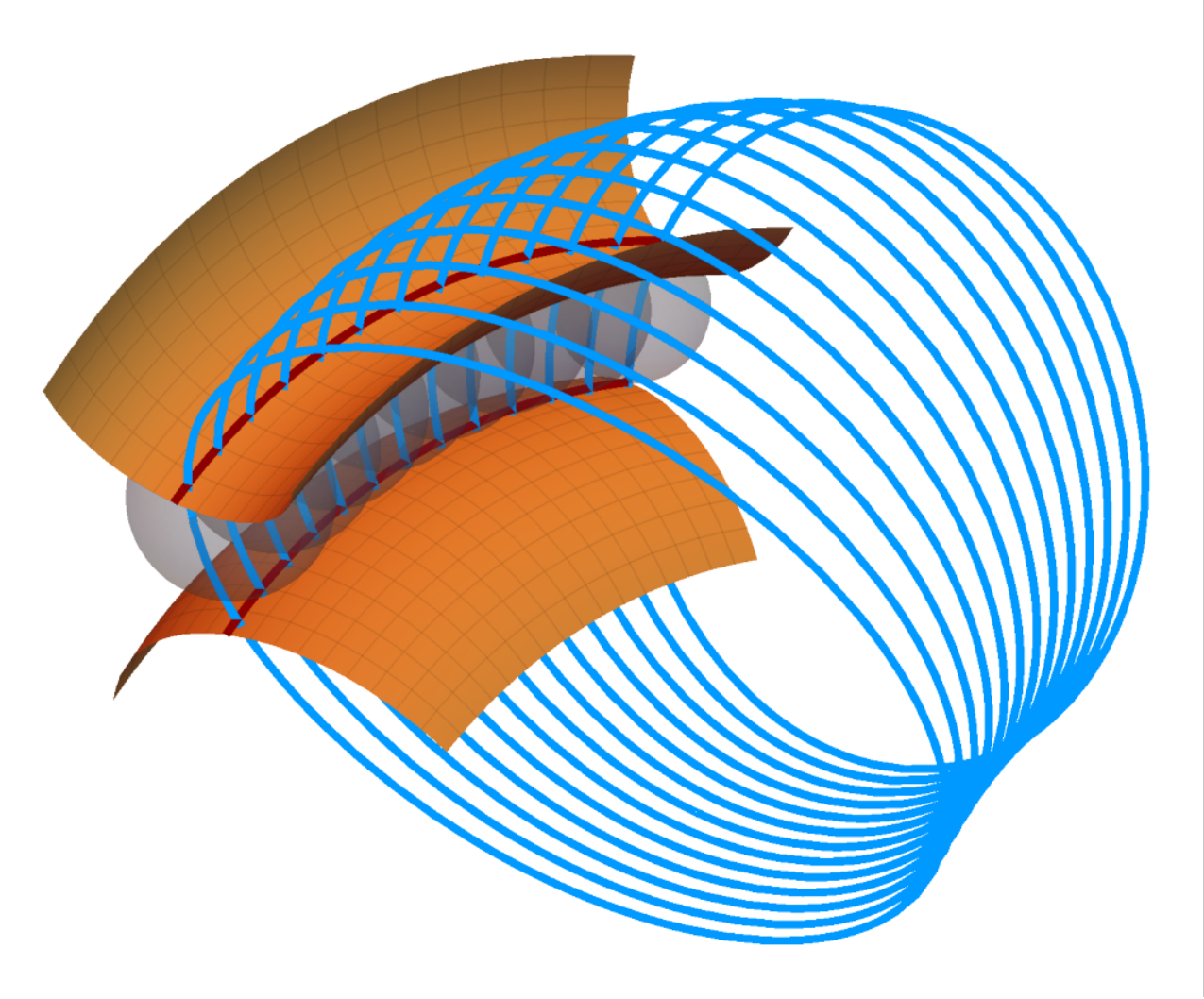}
\caption{\textit{Left.} The Ribaucour spheres (gray) along two corresponding curvature circles of Ribaucour related Dupin cyclides (orange) are curvature spheres of a Dupin cylide, namely of the Ribaucour cyclide. \textit{Right.} The circles of the associated cyclic circle congruence (blue) along two corresponding curvature lines then form a further Dupin cyclide; it is a 2-ortho Dupin cyclide of the Ribaucour cyclide provided by the Ribaucour spheres.}\label{fig_orth_dupin_system}
\end{figure}

\noindent We remark that similar arguments lead to the special case discussed by Darboux \cite[Chap.\,I\!I\!I \S 35]{darboux_ortho}: any Ribaucour pair of a Dupin cyclide and a totally umbilic surface also gives rise to a DC-system.

However, in Subsection \ref{subsect_interpretation} we will see that there are Lam\'e families in DC-systems that do not arise from Darboux's construction.

\subsection{DC-systems via Lie sphere geometric evolution}\label{subsect_dc_via_evolution} In the last subsection we have derived DC-systems from a Ribaucour pair of Dupin cyclides. To gain an explicit construction and a complete characterization of those systems, below we discuss how DC-systems are obtained by a Lie sphere geometric evolution of an initial Dupin cyclide.
\\\\We start by reconsidering Ribaucour pairs of Dupin cyclides and take up a property implicitly already discussed in \cite[Thm 4.10]{rib_families}:
\begin{prop}\label{prop_dupin_rib}
Two Dupin cyclides form a Ribaucour pair if and only if they are related by a fixed Lie inversion, that is, two Dupin cyclides $\Delta$ and $\hat{\Delta}$ are Ribaucour related if and only if for a suitable labelling of the $(2,1)$-splittings $\Delta=D_1 \oplus_\perp D_2$ and $\hat{\Delta}=\hat{D}_1 \oplus_\perp \hat{D}_2$ there exists a fixed Lie inversion $\sigma$ such that $\hat{D}_1=\sigma(D_1)$ and $\hat{D}_2=\sigma(D_2)$ .
\end{prop}
\begin{proof}
Let $(u,v) \mapsto f(u,v) \in \mathcal{Z}$ be a Dupin cyclide parametrized along the curvature leaves and $\sigma_a$ a fixed Lie inversion. Since curvature directions are invariant under Lie inversions, the Dupin cyclide $\hat{f}(u,v):=\sigma_a(f(u,v))$ and $f$ form a Ribaucour pair that envelopes the sphere congruence~$f \cap a^\perp$. 
\\\\Conversely, suppose that the Dupin cyclides $\Delta$ and $\hat{\Delta}$ form a Ribaucour pair simultaneously parametrized by curvature line coordinates via $f, \hat{f}:M^2 \rightarrow \mathcal{Z}$. Furthermore, we denote the enveloped Ribaucour sphere congruence by $r(u,v):=f(u,v) \cap \hat{f}(u,v)$. We will show that all Ribaucour spheres lie in a 5-dimensional subspace of $\mathbb{R}^{4,2}$ and, thus, lie in a linear sphere complex. The Lie inversion with respect to this linear sphere complex then interchanges the two envelopes of the Ribaucour sphere congruence and therefore maps the two Dupin cyclides onto each other. 

To do so we fix a point $(u_0, v_0) \in M^2$. Then, by Corollary \ref{cor_r_cyclides_dupin}, the Ribaucour spheres along the two curvature directions passing through $(u_0, v_0)$,
\begin{equation*}
v \mapsto r(u_0,v) \ \ \text{and} \ \  u \mapsto r(u, v_0),
\end{equation*}
are curvature spheres of two constant Ribaucour cyclides that intersect in the sphere $r(u_0,v_0)$. Hence, these Ribaucour spheres lie in a 5-dimensional subspace of $\mathbb{R}^{4,2}$, which we denote by $\mathcal{R}$. 

Let $(\bar{u}, \bar{v}) \in M^2$ be another arbitrary point. Since two contact elements of a Dupin cyclide that lie on the same curvature circle share a common curvature sphere, the four contact elements $f(\bar{u}, \bar{v}), f(u_0, \bar{v}),f(\bar{u}, v_0)$ and $f(u_0,v_0)$ lie in a 4-dimensional subspace $\mathcal{S} \subset \mathbb{R}^{4,2}$. Analogously, the 4-dimensional subspace that contains the spheres of the contact elements $\hat{f}(\bar{u}, \bar{v}), \hat{f}(u_0, \bar{v}), \hat{f}(\bar{u}, v_0)$ and $\hat{f}(u_0,v_0)$ will be denoted by $\hat{\mathcal{S}}$. As a consequence, the four Ribaucour spheres in these contact elements satisfy
\begin{equation*}
r(\bar{u},\bar{v}), r(u_0, \bar{v}), r(\bar{u}, v_0), r(u_0,v_0) \in \mathcal{S} \cap \hat{\mathcal{S}}
\end{equation*}
and therefore lie in a 3-dimensional space. Thus, we conclude that also $r(\bar{u},\bar{v}) \in \mathcal{R}$. Since this holds for any $(\bar{u}, \bar{v}) \in M^2$, the claim is proven.
\end{proof}
\noindent We emphasize that the Lie inversion in Proposition \ref{prop_dupin_rib} is not necessarily an M-Lie inversion. Thus, in general, the Lie inversion maps contact elements of the two Ribaucour transformed Dupin cyclides onto each other, but does not interchange the point sphere maps of them.  
\\\\Proposition \ref{prop_dupin_rib} will be the key to the construction of DC-systems via Lie sphere geometric evolutions. Since the proposed evolution works for any two Ribaucour transformed Legendre maps that are related by a fixed Lie inversion, we will firstly discuss it for general Legendre maps.

Since we now turn again to the construction of orthogonal systems, we fix a point sphere complex~$\mathfrak{p}$ and construct the systems in the M\"obius geometry modelled on $\lspann{\mathfrak{p}}^\perp$. 
\\\\Thus, let $f:M^2 \rightarrow \mathcal{Z}$ be a Legendre map. Then any Lie inversion $\sigma_a$, $a\neq p$, gives rise to a Ribaucour transform $f_a:=\sigma_a(f)$ with enveloped Ribaucour sphere congruence $r:=f \cap a^\perp$. The orthogonal surfaces of the cyclic circle congruence associated to $(f, f_a)$ are obtained by the following Lie geometric evolution: let us consider the set
\begin{equation*}
\mathcal{C}_a := \{ \lspann{\mathfrak{b}} \subset \mathbb{R}^{4,2} \ | \ \mathfrak{b} \in \lspann{\mathfrak{a}, \mathfrak{p}} \ \text{and } \lspan{\mathfrak{b}, \mathfrak{b}} \neq 0 \}.
\end{equation*}
Then $\mathcal{C}_a \ni b \mapsto \sigma_{b}(f)$ yields a 1-parameter family of Legendre maps that are Ribaucour transforms of $f$. Moreover, those satisfy the sought-after property:
\begin{prop}\label{prop_orth}
Let $f$ and $\hat{f}:=\sigma_a(f)$ be a Ribaucour pair related by the Lie inversion $\sigma_a$. Then $\mathcal{C}_a \ni b \mapsto \sigma_{b}(f)$ yields the 1-parameter family of Legendre maps that are orthogonal to the cyclic circle congruence associated to $f$ and $\hat{f}$.
\end{prop}
\begin{proof}
Let $f: M^2 \rightarrow \mathcal{Z}$ be a Legendre map and $\sigma_a$, $a \neq p$, a fixed Lie inversion. Note that $f=\lspann{f^p,r}$, where $f^p$ denotes the point sphere map of $f$ and $r:=f \cap a^\perp$. Then, by (\ref{equ_orth_circle}), the circle congruence $\Gamma=(\gamma, \gamma^\perp): M^2 \rightarrow G_{2,1} \times G_{2,1}$ associated to the Ribaucour pair $f$ and $\hat{f}=\sigma_a(f)$ is described by the $(2,1)$-planes
\begin{equation*}
\gamma=\lspann{ \mathfrak{f}^p, \hat{\mathfrak{f}}^p, \mathfrak{r}+ \lspan{\mathfrak{r}, \mathfrak{p}}\mathfrak{p}},
\end{equation*}
where $\hat{\mathfrak{f}}^p$ denotes the point sphere map of $\hat{f}$. Since the spheres in the contact elements $f$ and $\hat{f}$ are orthogonal to the corresponding circle, we know that any sphere $v \in \gamma^\perp$ satisfies $\mathfrak{v}+\lspan{\mathfrak{v}, \mathfrak{p}}\mathfrak{p} \perp \mathfrak{f}, \hat{\mathfrak{f}}$. Hence, because $\hat{f} \in \lspann{\mathfrak{f}^p, \mathfrak{r}, \mathfrak{a}}$, we conclude that $\mathfrak{v}+\lspan{\mathfrak{v}, \mathfrak{p}}\mathfrak{p} \perp \mathfrak{a}$.

Considering the family $\mathcal{C}_a \ni b \mapsto \sigma_{b}(f)$ of Legendre maps, we note that the spheres in the contact elements of these Legendre maps lie in the subspaces $\mathcal{S}:=\lspann{\mathfrak{f}^p, \mathfrak{r}, \mathfrak{a}, \mathfrak{p}}$.  Thus, by above, we have that any sphere $v \in \gamma^\perp$ fulfills $\mathfrak{v}+\lspan{\mathfrak{v}, \mathfrak{p}}\mathfrak{p} \in \mathcal{S}^\perp$ and the claim is proven. 
\end{proof}
\noindent If we restrict this construction and choose Dupin cyclides as initial Legendre maps, then the proposed evolution provides DC-systems: 
\begin{thm} \label{thm_dc_via_evolution}
Let $\Delta=D_1 \oplus_\perp D_2$ be a Dupin cyclide and $\sigma_a$ a Lie inversion, $a \neq p$. Then the 1-parameter family $\mathcal{C}_a \ni b \mapsto \sigma_b(\Delta)$ of Dupin cyclides provides a Lam\'e family of a DC-system. If there exists a sphere in $\lspann{\mathfrak{a}, \mathfrak{p}}$, then this sphere as totally umbilic Legendre map is also part of the Lam\'e family.

Moreover, any Lam\'e family of Dupin cyclides that provides a DC-system is  obtained in this way.
\end{thm}
\begin{proof}
From Proposition \ref{prop_orth}, it follows that $\mathcal{C}_b \ni b \mapsto \sigma_b(\Delta)$ yields a family of Dupin cyclides that are orthogonal surfaces to the circle congruence associated to the Ribaucour pair $\Delta$ and $\hat{\Delta}:=\sigma_a(\Delta)$. Since therefore all curvature lines of the orthogonal surfaces are circular, the curvature lines of the coordinate surfaces that complete the Lam\'e family to a triply orthogonal system are also circular. Thus, this construction indeed yields a DC-system.

Furthermore, according to Proposition \ref{prop_dupin_rib}, Theorem \ref{thm_dc_via_cyclic} and the well-known fact that any two orthogonal surfaces of a cyclic circle congruence form a Ribaucour pair, any DC-system can be obtained by this construction.
\end{proof}
%
%
\subsection{DC-systems in space forms}\label{subsect_interpretation}
In the previous subsections we have constructed DC-systems in a M\"obius geometry modelled on $\lspann{\mathfrak{p}}^\perp$. To gain further insights into their geometry, in this subsection we will further break symmetry and consider the DC-systems in specific space forms. Depending on the point of view, we give two geometric interpretations of the obtained Lie geometric evolution: firstly, we will discuss that any Lam\'e family of a DC-system yields parallel surfaces in a suitably chosen space form. Secondly, to make contact with classical results, we will fix a flat space form (Euclidean 3-space) and will characterize various types of Lam\'e families in DC-systems.
\\\\Recall that any choice of space form vector $\mathfrak{q} \in \mathbb{R}^{4,2} \setminus \{ 0\}$, $\mathfrak{q} \perp \mathfrak{p}$, distinguishes a quadric 
\begin{equation*}
\mathcal{Q}:=\{ \mathfrak{n}\in \mathbb{R}^{4,2} \ | \ \lspan{\mathfrak{n}, \mathfrak{n}}=0, \ \lspan{\mathfrak{n}, \mathfrak{q}}=-1, \ \lspan{\mathfrak{n}, \mathfrak{p}}=0 \}
\end{equation*}
of constant sectional curvature $-\lspan{\mathfrak{q},\mathfrak{q}}$ and obtain its complex of hyperplanes
\begin{equation*}
\mathcal{H}:= \{\mathfrak{n}\in \mathbb{R}^{4,2} \ | \ \lspan{\mathfrak{n}, \mathfrak{n}}=0, \ \lspan{\mathfrak{n}, \mathfrak{p}}=-1, \ \lspan{\mathfrak{n}, \mathfrak{q}}=0   \}.
\end{equation*}
In Theorem \ref{thm_dc_via_evolution} we have proven that any Lam\'e family of a DC-system can be generated by a pair of initial data, namely a Dupin cyclide and a Lie inversion $\sigma_a$. Evolving the Dupin cyclide via Lie inversions with respect to vectors in $\mathcal{C}_a \subset \mathcal{C}:=\lspann{ \mathfrak{a}, \mathfrak{p}}$, then yields a 1-parameter family of Dupin cyclides that are part of a DC-system.   

Suppose that $\{\mathfrak{e}_1, \mathfrak{p} \}$ provides an orthonormal basis for the 2-dimensional subspace $\mathcal{C}=\lspann{\mathfrak{a}, \mathfrak{p}}$. If we now choose the space form vector $\mathfrak{q}:=\mathfrak{e}_1$, then the Lie geometric evolution yields a 1-parameter family of parallel Legendre maps in this specific space form. Otherwise said, the circles of the associated cyclic circle congruence become lines in the chosen space form $\mathcal{Q}$, that is, the space form vector $\mathfrak{q}$ lies in the $(2,1)$-plane that contains the circle points.

Hence, we have proven:
\begin{prop}\label{prop_parallel_spaceform}
Any Lam\'e family of a DC-system consists of parallel surfaces in an appropriate space form; thus, the Dupin cyclides are orthogonal to a normal line congruence in this space form. 
\end{prop}
%
\begin{figure}
\hspace*{0.1cm}\includegraphics[scale=0.18]{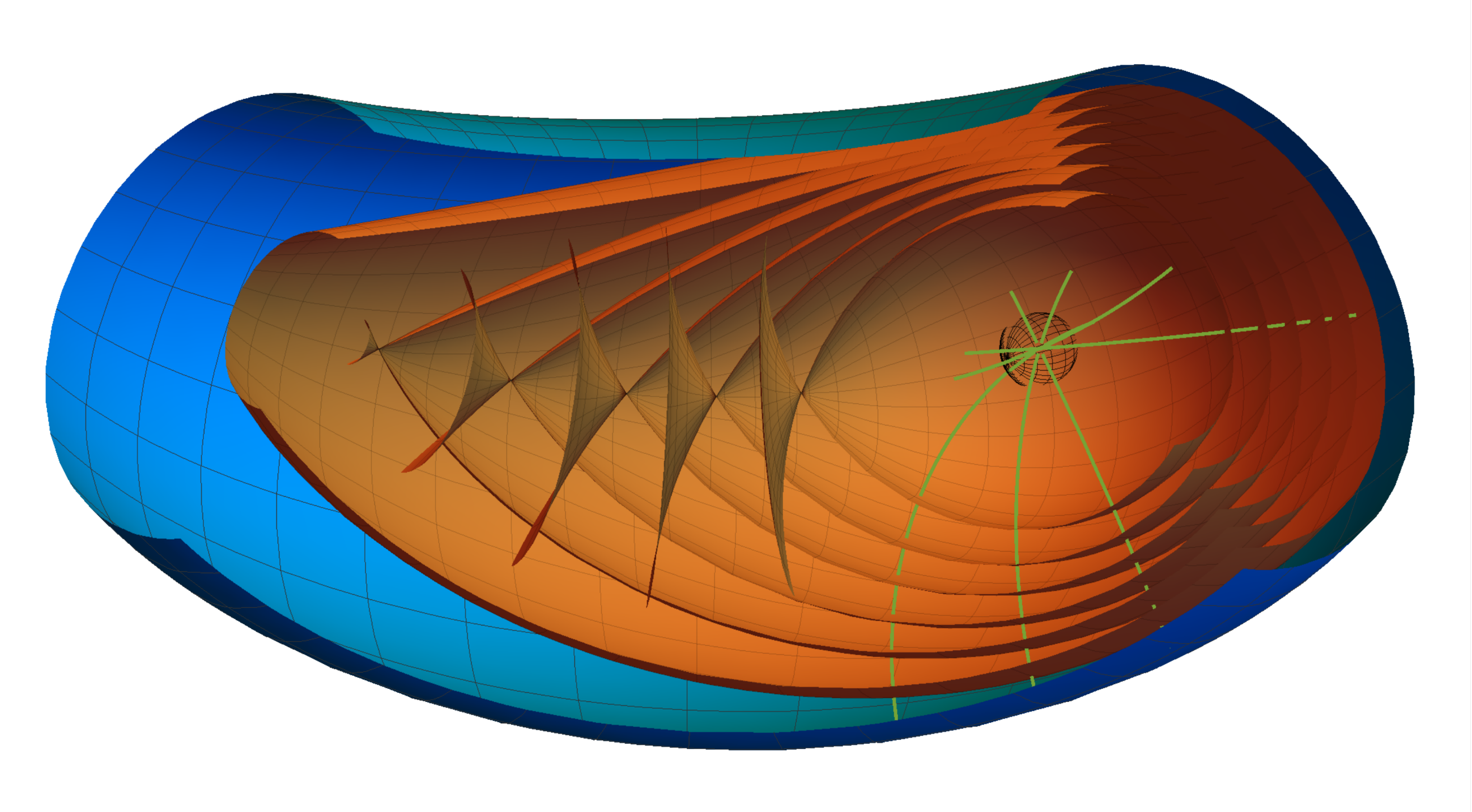}
\hspace*{0.1cm}\includegraphics[scale=0.38]{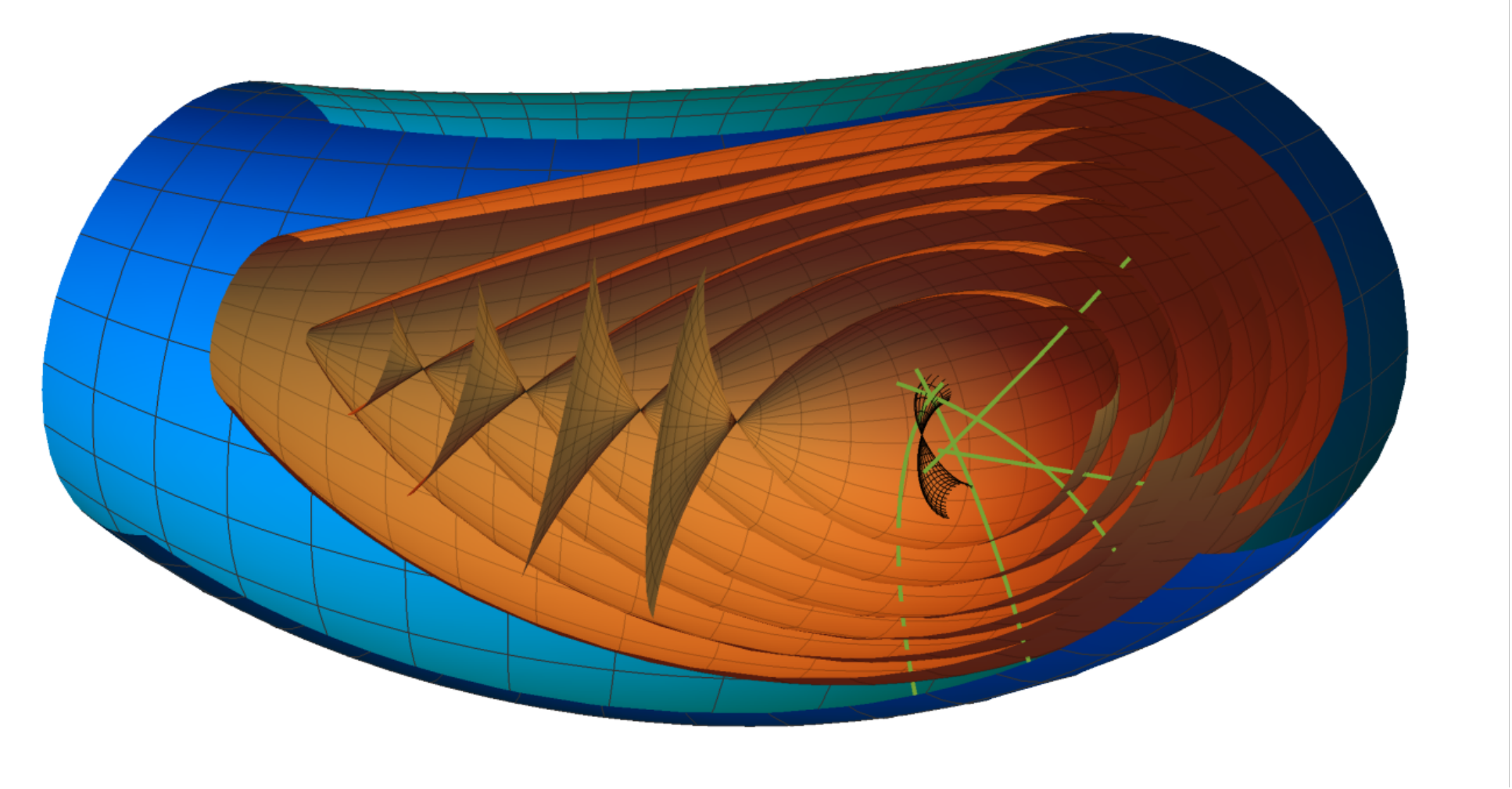}
\caption{These two Lam\'e families of DC-systems are generated by suitably evolving an initial Dupin cyclidic patch (blue). Depending on the chosen evolution, the orthogonal trajectories (green) intersect in one point (Type 1, \emph{left}) or do not intersect (Type 3, \emph{right}).
}
\label{fig_types}
\end{figure} 
%
\ \\As a consequence Lam\'e families of DC-systems split into three types. To link these types to well-known DC-systems, we now change our point of view, fix a flat space form and analyse them in Euclidean space. To do so, we choose a space form vector $\mathfrak{q} \in \mathcal{L}$, $\mathfrak{q}\perp \mathfrak{p}$.

Depending on the signature of the space $\mathcal{C}=\lspann{\mathfrak{e}_1, \mathfrak{p}}$, we distinguish the following three types (for examples see Figures \ref{fig_types} and \ref{fig_types_ell}):
\\\\\framebox{Type 1: $\lspan{\mathfrak{e}_1, \mathfrak{e}_1}=0$}
\\\\\textbf{Confocal Dupin cyclides, CFDs \cite{Honglawan34, Lavrentovich5, Schief2005OnAN}.} If $e_1=q$, then the Lam\'e family $\mathcal{C}_a \ni b \mapsto \sigma_b(\Delta)$ yields parallel Dupin cyclides  in the distinguished Euclidean space form. Otherwise said, we obtain a family of confocal Dupin cyclides. 

The two complementary Lam\'e families of these DC-systems are compound of circular cones built from the Euclidean normals along each circular curvature line of the confocal Dupin cyclides.
\ \\\\\textbf{All orthogonal trajectories have one point in common.} Otherwise, if $q\neq e_1$, then $e_1$ represents a point sphere that lies on all circles of the associated circle congruence. Thus, in this case all those circles intersect in one point.
\\\\\framebox{Type 2: $\lspan{\mathfrak{e}_1, \mathfrak{e}_1}>0$} 
\\\\\textbf{Lam\'e families of Dupin cyclides that contain a totally umbilic coordinate surface.} Since the spacelike vector $\mathfrak{e}_1$ satisfies $\lspan{\mathfrak{e}_1, \mathfrak{p}}=0$, the map $\sigma_{e_1}$ is a proper M-Lie inversion; hence, describes a reflection in the M-sphere $s^\pm \in \lspann{\mathfrak{e}_1, \mathfrak{p}} \cap \LL$. Thus, in the Lam\'e family $\mathcal{C}_a \ni b \mapsto \sigma_b(\Delta)$ two Dupin cyclides are related by this M\"obius inversion.
\\\\Note that for this type of Lam\'e family the evolution map has two singularities at the values
\begin{equation*}
 b^\pm := \mathfrak{e}_1 \pm\sqrt{\lspan{\mathfrak{e}_1, \mathfrak{e}_1}} \mathfrak{p} \in \lspann{\mathfrak{e}_1, \mathfrak{p}},
\end{equation*}
where the vectors $b^\pm$ become lightlike. However, the represented spheres $b^\pm$ that coincide up to orientation are orthogonal to the circle congruence. Therefore, seen as totally umbilic surfaces, those spheres complete the constructed Lam\'e family. 

We remark that these Lam\'e families are also obtained by Darboux's construction \cite[Chap.\,I\!I\!I~\S 35]{darboux_ortho} that result in Lam\'e families of Dupin cyclides with one totally umbilic surface.   

\ \\\\\framebox{Type 3: $\lspan{\mathfrak{e}_1, \mathfrak{e}_1}<0$} 
\\\\This type of Lam\'e family does not contain a totally umbilic surface or a point. Moreover, at first sight it is geometrically less lucid. After a stereographic projection to the 3-sphere, the circles of the constructed cyclic congruence are geodesics on the 3-sphere, hence great circles on it. 
%
\\\\\\Note that Lam\'e families of Type 2 and 3 obtain a M\"obius geometric symmetry: it is provided by the M-Lie inversion determined by the linear sphere complex $\mathfrak{e}_1 \in \lspann{\mathfrak{e}_1, \mathfrak{p}}$.

In particular, if $\sigma_{e_1}$ preserves quer-spheres of the Dupin cyclides, then those are totally umbilic coordinate surfaces in the complementary Lam\'e  families of the DC-system. This is due to the well-known fact that coordinate surfaces of different families intersect along curvature lines (for an example see Fig.\,\ref{fig_types_ell}, right). 

In particular, if $\sigma_{e_1}$ fixes an entire quer-sphere congruence, then we recover the special DC-systems already discussed in Subsection \ref{subsect_tos_2_ortho}.
\begin{figure}
\hspace*{-0.9cm}\vspace*{1.4cm}\includegraphics[scale=0.42]{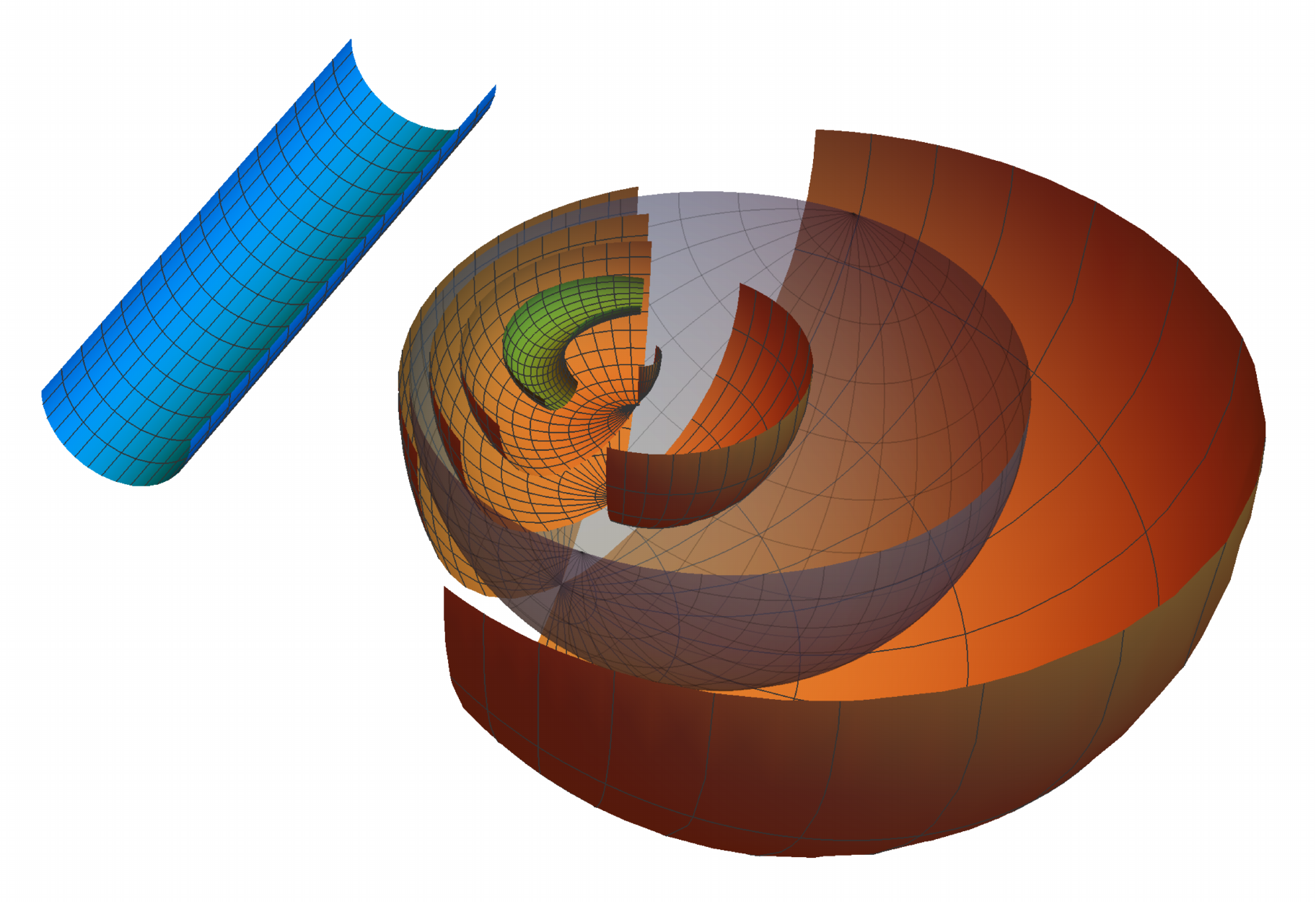}
\hspace*{1.4cm}\vspace*{-1.4cm}\includegraphics[scale=0.53]{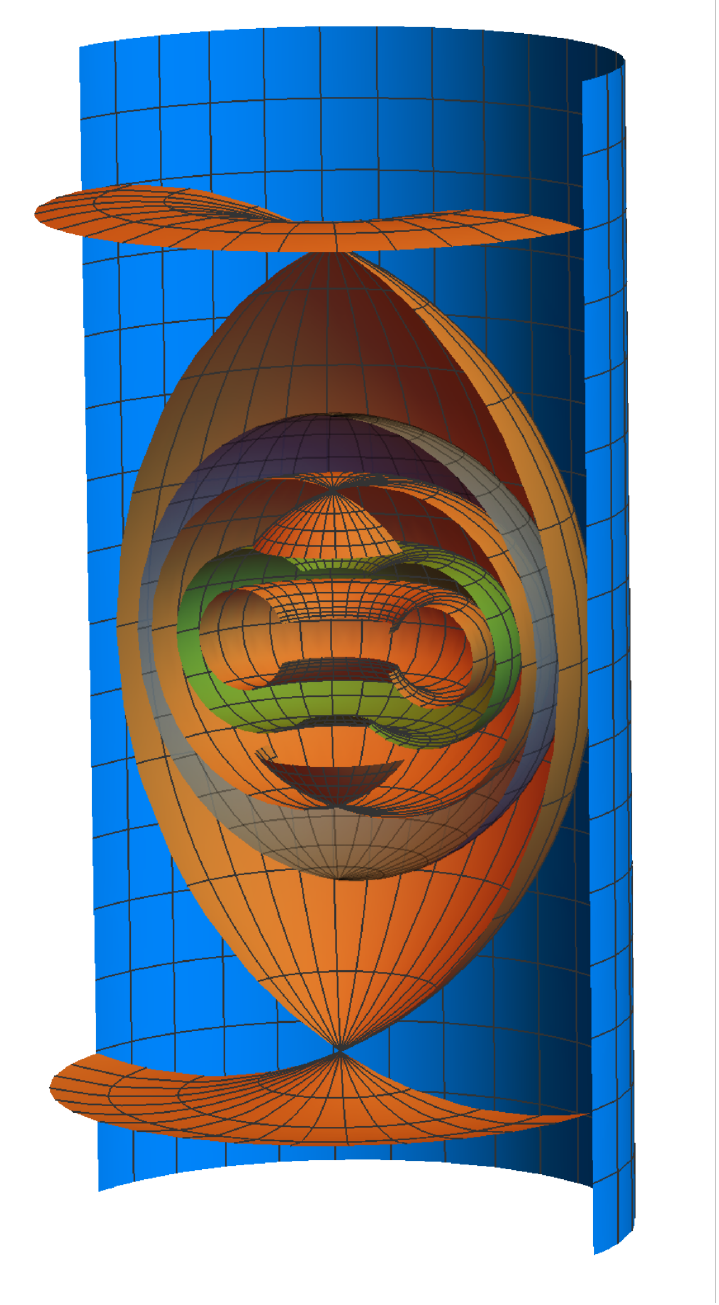}
\includegraphics[scale=0.45]{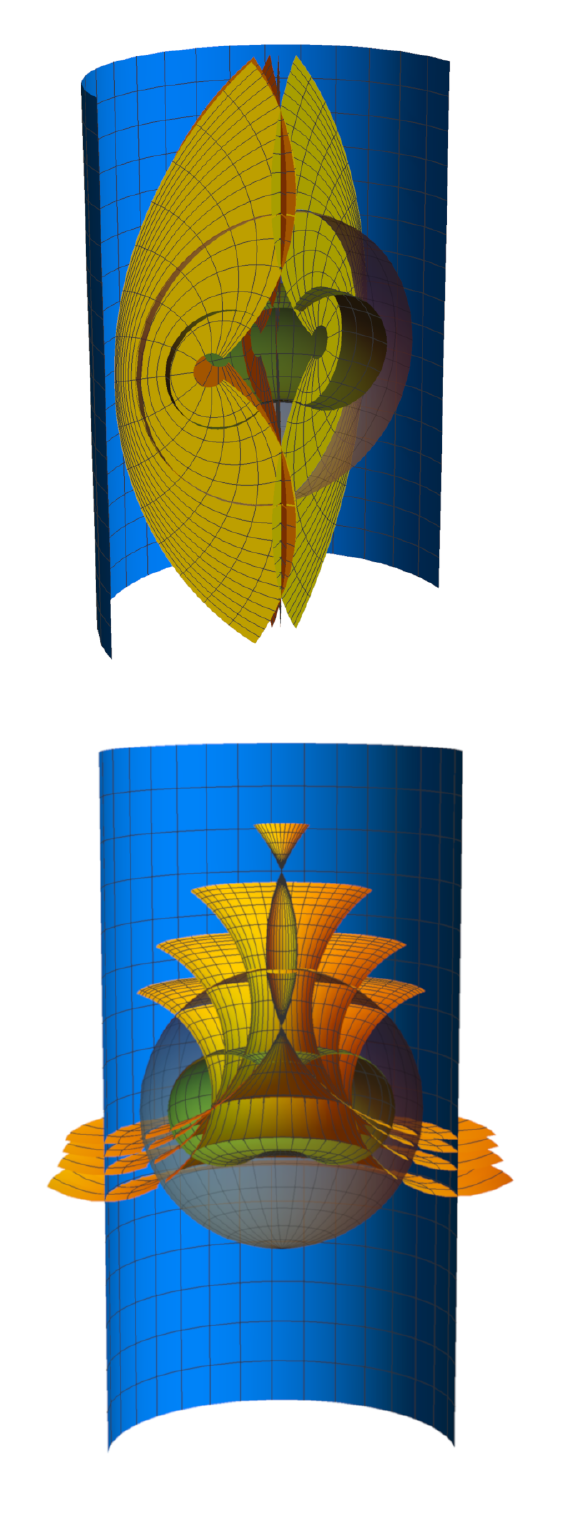}
\caption{\emph{Left and Middle.} Two examples of Lam\'e families of DC-systems obtained by evolving a cylinder (blue). Both families are of Type 2. The green Dupin cyclide patches are obtained by reflection in the gray sphere that is also part of the Lam\'e family. Depending on the position of this reflection sphere, the supplementary coordinate surface families can consist of totally umbilic coordinate surfaces.  \emph{Right.} The supplementary coordinate surfaces (yellow) of the Lam\'e family in the middle: one family is given by planes, namely the quer-spheres of the cylinder, the other one consists of Dupin cyclides.}
\label{fig_types_ell}
\end{figure} 
%
%
\section{Applications}\label{sect_applications}
\noindent In this last section we will briefly sketch some applications that could benefit from the developed approach to Dupin cyclides and DC-systems via Lie sphere geometric evolutions.
\subsection{Blending Dupin cyclides}
Due to their special geometric properties, patches of Dupin cyclides are popular as blending surfaces (see for example \cite{blending_dupin, blending_dupin_2, discrete_channel, PRATT1990221}). The discussed approach provides a direct and efficient way to construct these patches and allows direct conclusions on the geometry of the blend, as for example singularities. 

Blending Dupin cyclides can be recovered from a variety of different prescribed data. Here, as an example, we start with two curvature spheres and a curvature circle on one of the spheres as initial data and seek for the uniquely determined blending Dupin cyclide (cf.\,\cite{blending_dupin} and \cite[\S 2.3]{discrete_channel}). 
Thus, let $s_1 \in \LL$ and $s_2 \in \LL$ be two generic spheres and $t \mapsto \gamma_1(t) \in G_{(2,1)}^\mathcal{P}$ a parametrized circle that lies on~$s_1$. 

The following procedure will provide the missing data of the blending Dupin cyclide represented by $\Delta = D_1 \oplus_\perp D_2 \in G_{(2,1)} \times G_{(2,1)}$ with $s_1, s_2 \in D_1$:  
\begin{itemize}
\item consider the M-Lie inversion $\sigma_a$, where
\begin{equation*}
\mathfrak{a} = \lspan{\mathfrak{s}_2, \mathfrak{p}}\mathfrak{s}_1 - \lspan{\mathfrak{s}_1, \mathfrak{p}}\mathfrak{s}_2,
\end{equation*}
that interchanges data related to $s_1$ and $s_2$ (cf.\,Subsection~\ref{subset_param});
\vspace*{0.1cm}\item thus, the curvature circle $\gamma_2$ on $s_2$ is provided by
\begin{equation*}
t \mapsto \gamma_2(t):=\sigma_a(\gamma_1(t));
\end{equation*}
\item let $q_1 \in \LL$ be one of the two quer-spheres along $\gamma_1$, namely a  sphere that intersects $s_1$ along $\gamma_1$ orthogonally; since $\sigma_a$ is an M-Lie inversion, the sphere
\begin{equation*}
q_2 = \sigma_a(q_1)
\end{equation*}
intersects $s_2$ along $\gamma_2$ orthogonally and is therefore also a quer-sphere of the sought-after blending Dupin cyclide; 
\vspace*{0.1cm}\item the two quer-spheres $q_1$ and $q_2$ already determine an  M-sphere pencil with spheres in $\lspann{\mathfrak{q}_1, \mathfrak{q}_2, \mathfrak{p}}$; this pencil contains one family of quer-spheres of $\Delta$ and therefore give rise to the corresponding evolution map; 
\vspace*{0.1cm}\item hence, by using Proposition \ref{prop_evolution_quer1} and Corollary \ref{cor_evolution_quer2}, we can directly recover various missing data (e.\,g.\,curvature spheres, curvature circles) of the blending Dupin cyclidic patch.  
\end{itemize}

\subsection{Visualization of Dupin cyclides}
For the visualization of (patches of) Dupin cyclides it is often essential to have control over the curvature line parametrization and, in particular, on the placing of finitely many curvature lines. In the discussed framework, the placing amounts to the choice of finitely many M-Lie inversions in the evolution maps associated to the Dupin cyclide.

The following subdivision algorithm provides a possibility to arrange the curvature lines on a Dupin cyclidic patch in a symmetric way: let $\Delta=D_1 \oplus_\perp D_2$ be a Dupin cyclide and fix two curvature spheres $s_1, s_2 \in D_1$ that should restrict the patch in this coordinate direction; generically, this choice will split the Dupin cyclide into two patches.
\begin{itemize}
\item the M-Lie inversion $\sigma_a$, where
\begin{equation*}
\mathfrak{a} = \lspan{\mathfrak{s}_2, \mathfrak{p}}\mathfrak{s}_1 - \lspan{\mathfrak{s}_1, \mathfrak{p}}\mathfrak{s}_2,
\end{equation*}
maps the curvature circle $\gamma_1$ on $s_1$ onto the curvature circle $\gamma_2:=\sigma_a(\gamma_1)$ on $s_2$;
\vspace*{0.1cm}\item moreover, generically, there exist two curvature spheres $s_{12}, s_{12}' \in D_1$ that lie in the linear sphere complex $\mathbb{P}(\mathcal{L} \cap \{\mathfrak{a}\}^\perp)$; 
\vspace*{0.1cm}\item the curvature circles $\gamma_{12}=\sigma_{a_{12}}(\gamma_1)$ and $\gamma'_{12}=\sigma_{a'_{12}}(\gamma_1)$, where
\begin{equation*}
\mathfrak{a}_{12}= \lspan{\mathfrak{s}_{12}, \mathfrak{p}}\mathfrak{s}_1 - \lspan{\mathfrak{s}_1, \mathfrak{p}}\mathfrak{s}_{12} \ \ \text{and} \ \ \mathfrak{a}'_{12}= \lspan{\mathfrak{s}'_{12}, \mathfrak{p}}\mathfrak{s}_1 - \lspan{\mathfrak{s}_1, \mathfrak{p}}\mathfrak{s}'_{12},
\end{equation*}  
lie in a Lie geometric sense in the ``middle'' of $\gamma_1$ and $\gamma_2$; note that each of the two circles $\gamma_{12}$ and $\gamma'_{12}$ lie on one of the two distinguished patches; 
\vspace*{0.1cm}\item reapplying the described steps for the circle pairs $(\gamma_1, \gamma_{12})$ and $(\gamma_{12}, \gamma_2)$ will refine the curvature line net on the Dupin cyclidic patches; further iterative subdivision will then lead to the desired number of curvature lines; 
\vspace*{0.1cm}\item an analogous algorithm can be used to obtain the parameter lines of the other family of curvature lines.
\end{itemize}   
%
%
\subsection{Dupin cyclidic cubes}
Any DC-system restricted to a rectangular domain $[x_1,x_2] \times [y_1,y_2] \times [z_1,z_2]$ yields a Dupin cyclidic cube in the sense of \cite{dcyclidic}; that is, the eight corner point spheres
provide a circular hexahedron and the six Dupin cyclidic patches at the boundaries intersect orthogonally along common curvature circles.

The construction of DC-systems via evolution in Subsection \ref{subsect_dc_via_evolution} provides a method to construct and geometrically understand these Dupin cyclidic cubes. In particular, the various types of DC-systems can be exploited to obtain special configurations of the Dupin cyclidic cubes. For example, if two boundary Dupin cyclidic patches of opposite faces of the cube are related by an M-Lie inversion then, additionally, the opposite diagonal point spheres are also concircular (see Fig.\,\ref{fig_cubes}). 
\begin{figure}
\begin{minipage}{3cm}
\hspace*{-1.2cm}\includegraphics[scale=0.2]{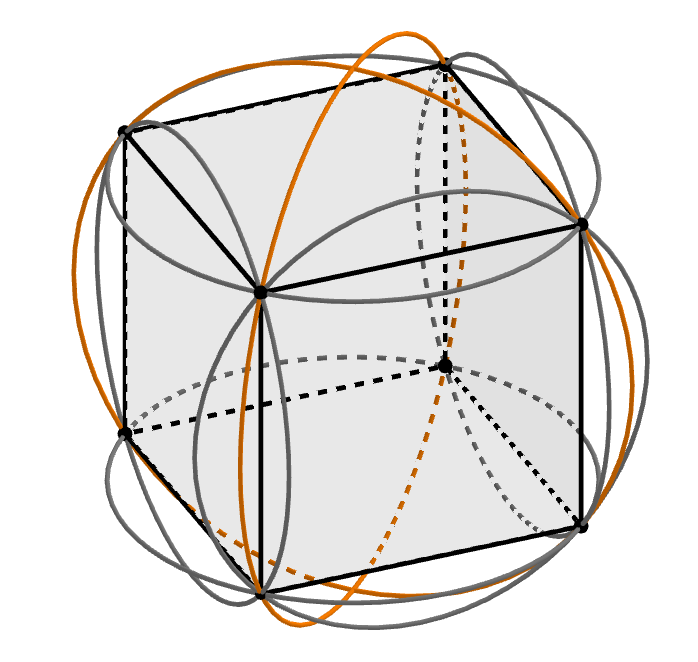}
\end{minipage}
\begin{minipage}{4cm}
\hspace*{-0.7cm}\includegraphics[scale=0.3]{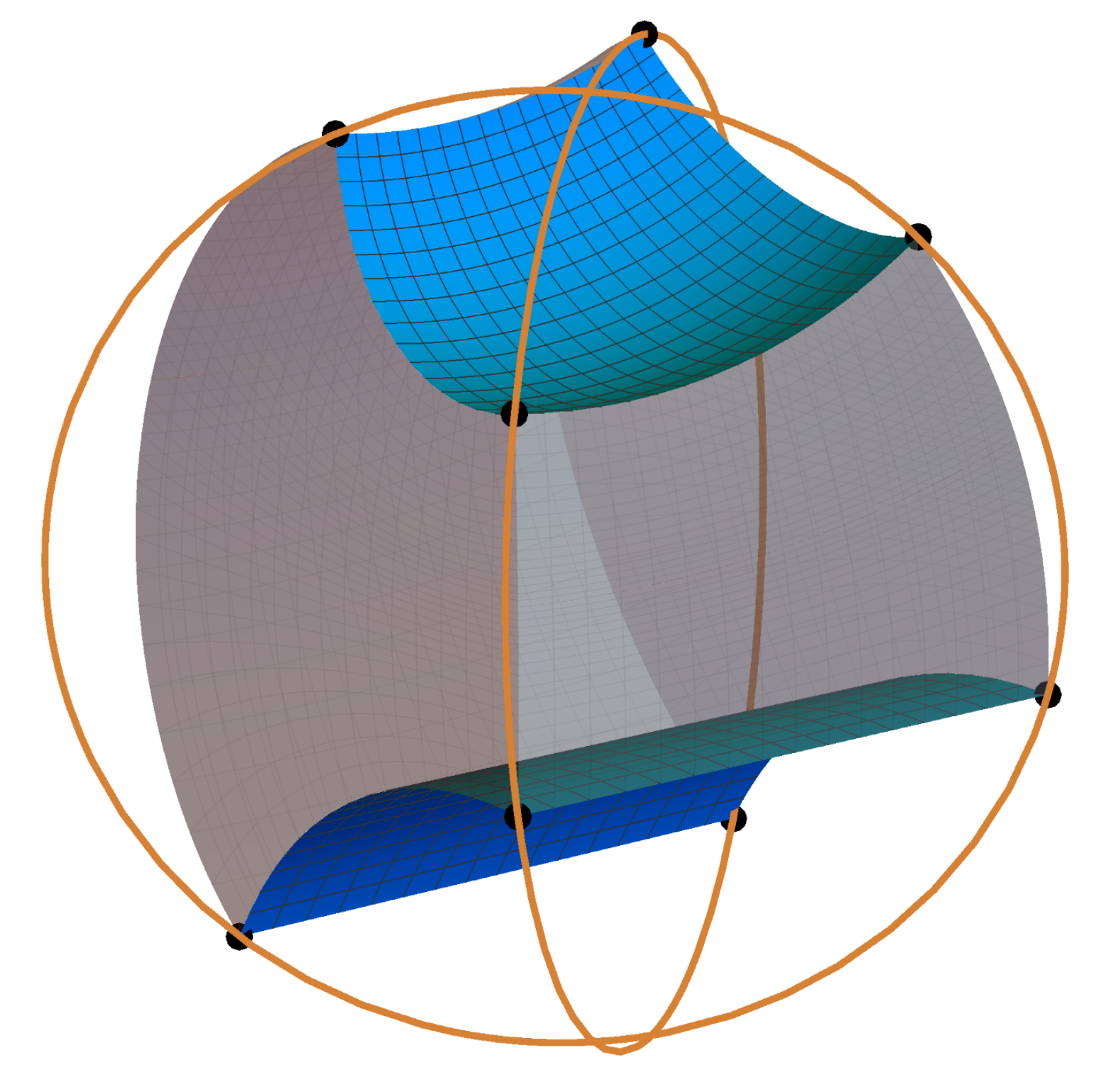}
\end{minipage}
\begin{minipage}{4cm}
\hspace*{0.7cm}\includegraphics[scale=0.4]{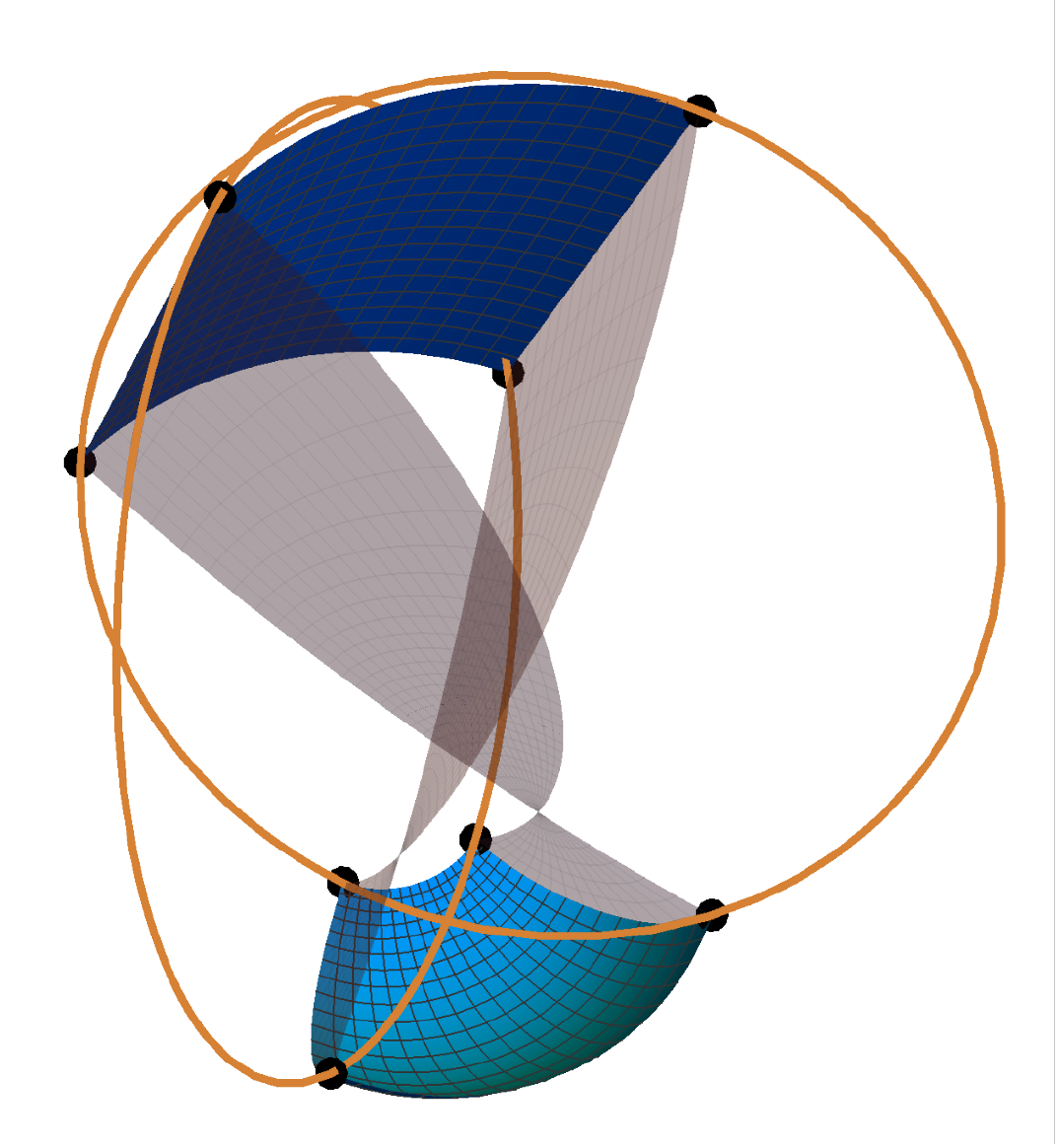}
\end{minipage}
\caption{Rectangular patches of DC-systems lead to 3D cyclidic nets \cite{dcyclidic}. If two opposite boundary patches are related by an M-Lie inversion, then opposite diagonal points of the underlying circular hexahedron also lie on a circle (orange). If the M-Lie inversion is elliptic, the vertical faces are embedded \textit{(middle)}, if it is hyperbolic they are non-embedded \textit{(right)}.}\label{fig_cubes}
\end{figure} 
\subsection{Discrete Dupin cyclidic systems} Discrete circular nets (that is, quadrilateral meshes such that the four vertices of each elementary quadrilateral lie on a circle) provide a possible discretization of curvature line parametrized surfaces (2-dimensional nets $\mathbb{Z}^2 \rightarrow \mathbb{R}^3$) or triply orthogonal systems (3-dimensional nets $\mathbb{Z}^3 \rightarrow \mathbb{R}^3$) \cite{ddg_book}.  

Note that four point spheres $x_1, x_2, \sigma(x_1), \sigma(x_2)$ that are related by an M-Lie inversion $\sigma$ lie in a 3-dimensional subspace of $\mathbb{R}^{4,2}$ and therefore lie on circle. As a consequence the proposed construction for Dupin cyclides and DC-systems interacts well with the discrete notion of circular meshes.

The choice of a discrete initial circle and finitely many M-Lie inversions in Proposition~\ref{prop_evolution_quer1} and Corollary~\ref{cor_evolution_quer2} gives rise to a discrete Dupin cyclide in the sense of~\cite[\S 2.4]{discrete_channel}. Furthermore, a discrete Dupin cyclide and finitely many Lie inversions in the evolution map discussed in Theorem~\ref{thm_dc_via_evolution} yield a Lam\'e family of discrete Dupin cyclides and  hereafter to discrete Dupin cyclidic systems. We remark that Lam\'e families of discrete DC-systems that contain a totally umbilic coordinate surface and the associated discrete cyclic systems were also obtained in \cite[\S 4.3]{hertrichjeromin2021discrete} with the help of discrete cyclic circle congruences.
%
%
\bibliography{mybib}

\vspace*{1cm}\begin{minipage}{12cm}
\textbf{Gudrun Szewieczek}, gudrun.szewieczek@tuwien.ac.at
\\TU Wien, Wiedner Hauptstra\ss e 8-10/104, 1040 Vienna, Austria
\end{minipage}
\end{document}